\documentclass[10pt,reqno]{amsart} 

\usepackage{amssymb,latexsym}
\usepackage{cite} 

\usepackage[height=190mm,width=130mm]{geometry} 

\theoremstyle{plain}
\newtheorem{theorem}{Theorem}

\newtheorem{proposition}{Proposition}

\theoremstyle{definition}
\newtheorem{definition}{Definition}

\theoremstyle{remark}

\newcommand{\Z}{\mathbb{Z}}
\newcommand{\C}{\mathbb{C}}
\newcommand{\R}{\mathbb{R}}

\newcommand{\vac}{\mathbf {1}_V}

\numberwithin{equation}{section} 

\newcommand{\N}{\ensuremath{\mathbb {N}}}

\newcommand{\g}{\ensuremath{\Gamma}}

\newcommand{\ps}{{\raise 1pt\hbox{\tiny (}}}

\newcommand{\pss}{{\raise 1pt\hbox{\tiny [}}}
\newcommand{\pdd}{{\raise 1pt\hbox{\tiny ]}}}
\newcommand{\pd}{{\raise 1pt\hbox{\tiny )}}}

\newcommand{\bs}{{\raise 1pt\hbox{\tiny [}}}
\newcommand{\bd}{{\raise 1pt\hbox{\tiny ]}}}

\def\cross{\mathinner{\mathrel{\raise0.8pt\hbox{$\scriptstyle>$}}
                 \joinrel\mathrel\triangleleft}}

\usepackage{citehack}
\usepackage{amsmath,amsthm}
\usepackage{amsfonts}
\usepackage{amssymb}
\usepackage{longtable}
\usepackage[matrix,arrow,curve]{xy}

\usepackage{lipsum}

\def\K{\mathcal{K}}

\usepackage{stackrel}

\newcommand{\be}{\begin{equation}}
\newcommand{\ee}{\end{equation}}

\newcommand{\nn}{\nonumber \\}

 \newcommand{\res}{\mbox{\rm Res}}

\newcommand{\wt}{\mbox{\rm wt}\ }

\newcommand{\one}{\mathbf{1}_V}

\newcommand{\nc}{\newcommand}
\nc{\cali}{\mathcal}
\nc{\on}{\operatorname}
\nc{\Wick}{{\mb :}}

\nc{\ddz}{\frac{\partial}{\partial z}}
\nc{\ch}{\mbox{ch}}
\nc{\Oo}{{\cali O}}
\nc{\cond}{|\,}
\nc{\bib}{\bibitem}
\nc{\pone}{\Pro^1}
\nc{\pa}{\partial}
\nc{\arr}{\rightarrow}
\nc{\larr}{\longrightarrow}
\nc{\ket}{\rangle}
\nc{\bra}{\langle}
\nc{\gam}{\bar{\gamma}}
\nc{\q}{\widetilde{Q}}
\nc{\ep}{\epsilon}
\nc{\su}{\widehat{{\mf s}{\mf l}}_2}
\nc{\sw}{{\mf s}{\mf l}}
\nc{\h}{{\mf h}}
\nc{\n}{{\mf n}}
\nc{\ab}{\mf{a}}
\nc{\is}{{\mb i}}
\nc{\js}{{\mb j}}
\nc{\bi}{\bibitem}
\nc{\He}{{\cali H}}
\nc{\inv}{^{-1}}
\nc{\ol}{\overline}
\nc{\wh}{\widehat}
\nc{\dst}{\displaystyle}

\nc{\delt}{\partial_t}
\nc{\ddt}{\frac{\partial}{\partial t}}
\nc{\delx}{\partial_x}
\nc{\mb}{\mathbf}
\nc{\mf}{\mathfrak}

\nc{\mbb}{\mathbb}
\nc{\Ctt}{\C((t))}
\nc{\Ct}{\C[t,t\inv]}

\nc{\ghat}{\wh{\g}}

\nc{\un}{\underline}
\nc{\mc}{\mathcal}
\nc{\BB}{{\mc B}}
\nc{\bb}{{\mf b}}
\nc{\kk}{{\mf k}}
\nc{\frob}{\times}
\nc{\sm}{\setminus}
\nc{\Pp}{{\mathbb P}^1}
\nc{\Aa}{{\mc A}}

\nc{\AutO}{\on{Aut}\Oo}
\nc{\AUTO}{\un{\on{Aut}}\Oo}
\nc{\AUTK}{\un{\on{Aut}}\K}
\nc{\Heout}{\He_{\out}}
\nc{\Hetil}{{\widetilde\He}}
\nc{\wb}{\overline}

\nc{\Res}{\on{Res}}
\nc{\pitil}{\Pi}
\nc{\Ctil}{\wt{C}}
\nc{\auto}{\on{Aut} \Oo}
\nc{\phitil}{\wt{\phi}}
\nc{\gz}{\g_{\vec z}}
\nc{\tensorM}{\bigotimes_{i=1}^N{\mathbb M}_i}
\nc{\tensorW}{\bigotimes_{i=1}^N W_{\nu_i,k}}
\nc{\out}{\on{out}}

\nc{\m}{{\mathfrak m}}

\nc{\gx}{\g^0_{\vec x}}

\nc{\hx}{\He^0_{\vec x}}
\nc{\tensorpi}{\pi_{\nu_1,\ldots,\nu_N}^\kappa}
\nc{\Phizw}{\Phi_{\vec w}({\vec z})}
\nc{\Pro}{{\mathbb P}}

\nc{\De}{\Delta}

\nc{\us}{\underset}

\nc{\Ll}{\mc L}
\nc{\dR}{\on{dR}}

\nc{\T}{{\mc T}}

\nc{\Xn}{\overset{\circ}X{}^n} \nc{\Dn}{\overset{\circ}D{}^n}
\nc{\Dxn}{\overset{\circ}D{}^n_x} \nc{\varphitil}{\wt{\varphi}}

\nc{\lf}{{\mf l}}
\nc{\GL}{{}^L G}
\nc{\Vir}{\on{Vir}}

\begin{document}
\title[Bigraded differential algebra for vertex algebra complexes]  
{Bigraded differential algebra for vertex algebra complexes} 
\author{A. Zuevsky} 
\address{Institute of Mathematics \\ Czech Academy of Sciences\\ Praha}

\email{zuevsky@yahoo.com}
\begin{abstract}
For an infinite chain bicomplex 
we show that the orthogonality and grading conditions 
provide it 
 with the structure of a bigraded differential algebra with respect to a natural multiplication
of several elements bicomplex spaces. 
Corresponding bigraded algebra commutation relations generate 
a sequence of non-vanishing cohomology invariants associated to vertex algebras. 
In particular, we apply this result to the bicomplex of 
 grading-restricted vertex algebra cohomology endowed with a multiplication we introduce. 
We provide examples associated to various choices of vertex algebra bicomplex subspaces. 
The generators and commutation relations of the bigraded differential algebra form a continual Lie algebra 
with the root space provided by a grading-restricted vertex algebra.   
AMS Classification: 53C12, 57R20, 17B69 
\end{abstract}

\keywords{Cohomological invariants; orthogonality condition for chain complexes; 
 bidifferential algebras; continual Lie algebras} 
\vskip12pt  

\maketitle
\section{Conflict of Interest Statement}
The author states that: 

1.) The paper does not contain any potential conflicts of interests. 

\section{Data availability statement}
The author confirms  that: 

\medskip 
1.) The paper does not use any datasets. No dataset were generated during and/or analysed 
during the current study. 

2.) The paper includes all data generated or analysed during this study. 

3.) Data sharing not applicable to this article as no datasets were generated or analysed during the current study.

4.) The data of the paper can be shared openly.  

\section{Introduction} 
\label{valued}
The cohomology theory for vertex operator algebras is an important and attractive theme for studies.
In \cite{Huang} the cohomology theory for a grading-restricted vertex algebra \cite{K}  
was introduced. 
The definition of bicomplex spaces and coboundary operators 
uses an interpretation of vertex algebras in terms of rational functions constructed from    
 matrix elements \cite{H2, FQ} for a grading-restricted vertex algebra.  
The notion of composability with a number of vertex operators for bicomplex  
space elements is essentially  
involved in the formulation. 
The cohomology of such complexes defines a cohomology of a grading-restricted vertex 
algebras in the standard way. 
 It is an important problem to study possible cohomological classes for vertex algebras.

For differential forms considered on smooth manifolds,  
the Frobenius theorem for a distribution leads to the orthogonality condition.  
In this paper we show that the orthogonality condition with respect to a commutator, 
and double grading conditions assumed for 
elements of the bicomplex spaces associated to a grading-restricted vertex algebra 
endow the bicomplex spaces with 
the structure of 
a bigraded differential algebra with respect the commutator of bicomplex mappings. 
 The orthogonality condition for elements of bicomplex spaces 
is motivated by geometrical construction of cohomological invariants 
for foliated manifolds \cite{Ghys}. 
 Originally, the orthogonality comes from consideration of differential forms. 
 Here we implement similar idea in the case of rational functions associated to 
grading-restricted vertex algebras. 
Using the above mentioned conditions we then find further explicit examples of continual Lie algebras 
\cite{saver} associated to vertex algebras. 
 We derive also cohomological classes for the bicomplex for a grading-restricted vertex algebra.  
 Such cohomological classes are non-vanishing and independent of the choice of the bicomplex space elements.    
Our main motivation was to show that the bicomplex construction
 originating from algebraic properties of vertex algebras, and formulated in terms of rational functions 
with specific properties possesses deeper cohomological structure similar to that of differential forms 
in algebraic topology. 

As for possible applications of the material presented in this paper, we would like to mention 
computations of higher cohomologies for grading-restricted vertex algebras \cite{Hu3}, 
search for more complicated 
cohomological invariants, and applications to differential geometry. 
In particular, since vertex algebras is a useful computational tool, 
it would be interesting to study possible relations to cohomology of manifolds. 
One can show that such cohomological invariants possess analytical (with respect 
to the notion of composability) as well as geometrical meaning. 
In addition to the natural orthogonality condition, ona can consider variations of multiplications 
defined for bicomplex spaces, and, therefore, more advanced examples of graded differential 
algebras. 
In differential geometry there exist various approaches to the construction of cohomological classes
(cf., in particular, \cite{Losik}). We hope to use these techniques to derive counterparts in 
the cohomology theory of vertex algebras.
The results proven in this paper are also useful in computations of cohomology of foliations \cite{BG, BGG}. 

\section{$\overline{W}$-valued rational functions} 
Recall the notion of a grading restricted vertex algebras described in Appendix \ref{grading}.  
In this Section we recall the notion of a special rational functions 
which form spaces for the chain-cochain construction.  
Let $V$ be a grading-restricted 
vertex algebra, and $W$ a grading-restricted generalized $V$-module.   
 One defines the configuration spaces    
\[
F_{n}\C=\{(z_1, \dots, z_n)\in \C^n \;|\; z_i \ne z_j, i\ne j\}, 
\] 
for $n\in \Z_+$.
By $\overline{W}$ we denote the algebraic completion of $W$, 
\[
\overline{W}=\prod_{n\in \mathbb C}W_{(n)}=(W')^*.
\]
 A $\overline{W}$-valued rational function in $(z_1, \dots, z_n)$  
with the only possible poles at 
$z_i=z_j$, $i\ne j$, 
is a map 
\begin{eqnarray*}
 f: F_n \C &\to& \overline{W},   
\\
 (z_1, \dots, z_n) &\mapsto& f(z_1, \dots, z_n),   
\end{eqnarray*} 
such that for any $w'\in W'$,  
\begin{equation}
\label{pochta}
\langle w', f(z_1, \dots, z_n) \rangle, 
\end{equation} 
is a rational function in $(z_1, \dots, z_n)$  
with the only possible poles at 
$z_i =z_j$, $i\ne j$. 
Such map one calls $\overline{W}$-valued rational function in
$(z_1, \dots, z_n)$ with possible other poles. 
Denote the space of all $\overline{W}$-valued rational functions in
$(z_1, \dots, z_n)$ by $\overline{W}_{z_1, \dots, z_n}$. 
One defines a left action of $S_{n}$ on $\overline{W}_{z_1, \dots, z_n}$
by
\[
(\sigma(f))(z_1, \dots, z_n)=f(z_{\sigma(1)}, \dots, z_{\sigma(n)}),
\]
for $f\in \overline{W}_{z_1, \dots, z_n}$. 
\subsection{The requirements for $\overline{W}$-valued functions}
In \cite{Huang} the following properties for linear maps 
\[
\Phi: V^{\otimes n}\to  
\overline{W}_{z_1, \dots, z_n}, 
\]
are required in order to be able to construct the bicomplexes of $\overline{W}_{z_1, \dots, z_n}$-valued functions. 
Such a linear map is called to have
the $L(0)$-conjugation property if for $(v_1, \dots, v_n) \in V$, 
$w'\in W'$, $(z_1, \dots, z_n)\in F_n\C$ and $z\in \C^{\times}$, so that 
$(zz_1, \dots, zz_n)\in F_n\C$,
\begin{eqnarray}
\label{loconj}
\lefteqn{\langle w', z^{L_W(0)}
(\Phi(v_1 \otimes \cdots \otimes v_n))(z_1, \dots, z_n)\rangle}\nn
&&=\langle w', 
(\Phi(z^{L(0)}v_1 \otimes \cdots \otimes  z^{L(0)}v_n))(zz_1, 
\dots, zz_n)\rangle.
\end{eqnarray}

For $n\in \Z_+$, a linear map $\Phi$ 
 is called to have
the $L(-1)$-derivative property if (i) 
\begin{eqnarray}
\label{lder1}
\lefteqn{\frac{\partial}{\partial z_i}\langle w', 
(\Phi(v_1 \otimes \cdots  \otimes v_n))(z_1, 
\dots, z_n)\rangle}\nn
&&=\langle w', 
(\Phi(v_1 \otimes \cdots \otimes v_{i-1}\otimes L_V(-1)v_i
\otimes v_{i+1}\otimes \cdots \otimes v_n))(z_1, \dots, z_n)\rangle, 
\end{eqnarray}
for $i=1, \dots, n$, $(v_1, \dots, v_n)\in V$, and
$w'\in W'$ and (ii) 
\begin{eqnarray}
\label{lder2}
\lefteqn{\left(\frac{\partial}{\partial z_1}+\cdots +\frac{\partial}{\partial z_n}\right)
\langle w', 
(\Phi(v_1 \otimes \cdots \otimes v_n))(z_1, \dots, z_n)\rangle}
\nn
&&\quad=\langle w', L_W(-1)
(\Phi(v_1 \otimes \cdots  \otimes v_n))(z_1,
\dots, z_n)\rangle, 
\end{eqnarray}
and $(v_1, \dots, v_n)\in V$, 
$w'\in W'$.

One defines an action of $S_n$ on the space of linear maps  $\Phi$ by 
\[ 
(\sigma(\Phi))(v_1 \otimes \cdots\otimes v_n)
=\sigma(\Phi(v_{\sigma(1)}\otimes \cdots\otimes v_{\sigma(n)})),
\]  
for $\sigma\in S_n$ and $(v_1$, $\dots$, $v_n) \in V$.
We will use the notation $\sigma_{i_1, \dots, i_n}\in S_n$, 
to denote the 
the permutation given by $\sigma_{i_1, \dots, i_n}(j)=i_j$,  
for $j=1, \dots, n$.
\subsection{Matrix elements for coboundary operators}
As we will see in Section \ref{complexes}, the coboundary operators 
were introduced in terms of $E$-elements defined as follows. 
 For $w\in W$, the $\overline{W}$-valued function 
$E^{(n)}_W(v_1 \otimes \cdots\otimes v_n; w)$ is defined by  
\[
E^{(n)}_W(v_1 \otimes \cdots\otimes v_n; w)(z_1, \dots, z_n)
=E(Y_W(v_1, z_1)\cdots Y_W(v_n, z_n)w),  
\]
where an element $E(.)\in \overline{W}$ is 
given by 
\[
\langle w',E(.)\rangle =R(\langle w', . \rangle). 
\]
If a meromorphic function $f(z_1, \ldots , z_n)$ on a region in $\C$ which can be 
analytically extended to a rational function in $(z_1, \ldots, z_n)$, we denote 
such rational functions by $R(f(z_1, \ldots , z_n))$.   
We introduce 
\[
E^{W; (n)}_{WV}(w; v_1\otimes \cdots\otimes v_n) 
=E^{(n)}_W(v_1 \otimes \cdots\otimes v_n; w),
\]
where 
$E^{W; (n)}_{WV}(w; v_1 \otimes \cdots\otimes v_n)$ is  
an element of $\overline{W}_{z_1, \dots, z_n}$.
Next, we define 
\[
\Phi\circ \left(E^{(l_1)}_{V;\;\one}\otimes \cdots \otimes E^{(l_n)}_{V;\;\one}\right): 
V^{\otimes m+n}\to 
\overline{W}_{z_1,  \dots, z_{m+n}},
\] 
by
\begin{eqnarray*}
\lefteqn{(\Phi\circ (E^{(l_1)}_{V;\;\one}\otimes \cdots \otimes 
E^{(l_n)}_{V;\;\one}))(v_1 \otimes \cdots \otimes v_{m+n-1})}
\nn
&&=E(\Phi(E^{(l_1)}_{V; \one}(v_1 \otimes \cdots \otimes v_{l_1})\otimes \cdots
E^{(l_n)}_{V; \one}
(v_{l_1 +\cdots +l_{n-1}+1}\otimes \cdots 
\otimes v_{l_1 +\cdots +l_{n-1}+l_n}))),  
\end{eqnarray*}
and 
\[
E^{(m)}_W \circ_{m+1}\Phi: V^{\otimes m+n}\to 
\overline{W}_{z_1, \dots, z_{m+n-1}},
\]
 which is given by 
\begin{eqnarray*}
(E^{(m)}_W \circ_{m+1}\Phi)(v_1 \otimes \cdots \otimes v_{m+n})
=E(E^{(m)}_W(v_1 \otimes \cdots\otimes v_m;
\Phi(v_{m+1}\otimes \cdots\otimes v_{m+n}))). 
\end{eqnarray*}
Let us also introduce 
\[
E^{W; (m)}_{WV}\circ_0 \Phi: V^{\otimes m+n}\to  
\overline{W}_{z_1, \dots, z_{m+n-1}},
\]
 which is defined by 
\begin{eqnarray*}
(E^{W; (m)}_{WV}\circ_0 \Phi)(v_1 \otimes \cdots \otimes v_{m+n})
 =E(E^{W; (m)}_{WV}(\Phi(v_1 \otimes \cdots\otimes v_n)
; v_{n+1}\otimes \cdots\otimes v_{n+m})). 
\end{eqnarray*}
For 
\[
l_1 =\cdots=l_{i-1}=l_{i+1}=1,  
\]
\[
l_i=m-n-1,
\]
 for some $1 \le i \le n$, 
we write $\Phi\circ_i  E^{(l_i)}_{V;\;\one}$ for 
 $\Phi\circ (E^{(l_1)}_{V;\;\one}\otimes \cdots 
\otimes E^{(l_n)}_{V;\;\one})$.
\subsection{Maps composable with vertex operators}
\label{compos}
In order to increase or decrease the number of vertex operators included 
in a map $\Phi$, one has to assure convergence of resulting 
matrix elements. The composability notion 
determines the number of vertex operators that can be added to a map $\Phi$
by means of coboundary operators. 
For a $V$-module $W$, 
let 
\[
P_q: \overline{W}\to W_{(q)}, 
\]
 be the projection from $\overline{W}$ to $W_{(q)}$. 
For $m\in \N$, 
$\Phi$ is called to be composable with $m$ vertex operators 
when it satisfies: 
\begin{enumerate}
\item 
For $l_1, \dots, l_n \in \Z_+$, 
\[
\sum_{j=1}^nl_j=m+n,
\]
$(v_1, \dots, v_{m+n})\in V$ and $w'\in W'$, 
 introduce
\[
 \varphi_i
=
\varepsilon^{(l_i)}_V(v_{k_1} 
\otimes
\cdots \otimes v_{k_i} 
; \one)   
(z_{k_1}, 
 \dots, 
z_{k_i} 
),
\]
where 
\[
k_1=l_1+ \cdots +l_{i-1}+1, \; \; 
 ..., \; \;  v_{k_i}=l_1+\cdots +l_{i-1}+l_i,
\]
for $i=1, \dots, n$.
 Then it is assumed that there exist positive integers $N^n_m(v_i, v_j)$ 
depending only on $v_i$ and $v_j$ for $i, j=1, \dots, k$, $i\ne j$ such that 
\begin{equation}
\label{sumka}
\sum_{r_1, \dots, r_n\in \Z}\langle w', 
(\Phi(P_{r_1}\varphi_1 \otimes \dots\otimes
P_{r_n} \varphi_n)) (\zeta_1, \dots, 
\zeta_n)\rangle,
\end{equation} 
is absolutely convergent on the domains 
\[
|z_{l_1 +\cdots +l_{i-1}+p}-\zeta_i|  
+ |z_{l_1+\cdots +l_{j-1}+q}-\zeta_i|< |\zeta_i
-\zeta_j|,
\] 
for $i$, $j=1, \dots, k$, $i\ne j$, and for $p=1, 
\dots,  l_i$ and $q=1, \dots, l_j$. 
It is assumed that \eqref{sumka} is analytically extended to a
rational function
in $(z_1, \dots, z_{m+n})$, 
 independent of $(\zeta_1, \dots, \zeta_n)$, 
with the only possible poles at 
$z_i=z_j$, of order less than or equal to 
$N^n_m(v_i, v_j)$, for $i$, $j=1, \dots, k$,  $i\ne j$. 

\bigskip 
\item For $(v_1, \dots, v_{m+n})\in V$, it is assumed that there exist 
positive integers $N^n_m(v_i, v_j)$, depending only on $v_i$ and 
$v_j$, for $i$, $j=1, \dots, k$, $i\ne j$, such that for $w'\in W'$, 
\begin{equation}
\label{sumka1}
\sum_{q\in \C}\langle w', 
(E^{(m)}_W(v_1\otimes \cdots\otimes v_m; 
P_q((\Phi(v_{1+m}\otimes \cdots\otimes v_{n+m}))(z_{1+m}, \dots, z_{n+m})))\rangle, 
\end{equation}
is absolutely convergent when $z_i \ne z_j$, $i\ne j$
$|z_i|>|z_k|>0$ for $i=1, \dots, m$, and 
$k=m+1, \dots, m+n$. 
 It is assumed  that \eqref{sumka1} 
can be analytically extended to a
rational function 
in $(z_1, \dots, z_{m+n})$ with the only possible poles at 
$z_i=z_j$, of orders less than or equal to 
$N^n_m(v_i, v_j)$, for $i, j=1, \dots, k$, $i\ne j$,. 
\end{enumerate}
\section{Vertex algebra bicomplexes} 
\label{complexes}
In this Section we recal the notion of the chain-cochain bicomplex 
associated to a grading-restricted vertex algebra $V$. 
In Section \ref{invariants} we determine the structure of a bigraded differential algebra  
for these complexes. 

Let start with recall the definition of shuffles.   
For $l \in \N$ and $1\le s \le l-1$, let $J_{l; s}$ be the set of elements of 
$S_l$ which preserve the order of the first $s$ numbers and the order of the last 
$l-s$ numbers, i.e., 
\[
J_{l, s}=\{\sigma\in S_l \;|\;\sigma(1)<\cdots <\sigma(s),\;
\sigma(s+1)<\cdots <\sigma(l)\}.
\]
The elements of $J_{l; s}$ are called shuffles. Let us denote also 
\[
J_{l; s}^{-1}=\{\sigma\;|\; \sigma\in J_{l; s}\}. 
\]
The shuffles are used in the definition of $C^n_m(V, W)$ spaces in what follows,
 and play an essential role 
in the chain-cochain conditions \eqref{berta1}--\eqref{berta2}.   

Now we are on a position to define the space $C^n_m(V, W)$ 
of $\overline{W}_{z_1, \ldots, z_n}$-valued functions.  
\begin{definition}
For $n\in \Z_+$, let $C_0^n(V, W)$ be the vector space of all 
linear maps from $V^{\otimes n}$ to $\overline{W}_{z_1, \dots, z_n}$  
satisfying the $L(-1)$-derivative property and the $L(0)$-conjugation property. 
For $m$, $n\in \Z_+$, 
let $C_m^n(V, W)$ be the vector spaces of all  
linear maps from $V^{\otimes n}$ to $\overline{W}_{z_1, \dots, z_n}$ 
composable with $m$ vertex operators, and satisfying the $L(-1)$-derivative
property, the $L(0)$-conjugation property, and such that 
\begin{equation}
\label{shushu}
\sum_{\sigma\in J_{l; s}^{-1}}(-1)^{|\sigma|}
\sigma\left(\Phi(v_{\sigma(1)}\otimes \cdots \otimes v_{\sigma(l)})\right)=0.
\end{equation}
\end{definition}
Let us set $C_m^0(V, W)=W$. Then it was proven that  
\[
 C_m^n(V, W)\subset C_{m-1}^n(V, W), 
\]
for $m\in \Z_+$. 
For $\Phi \in C_m^n(V, W)$, 
 the coboundary operator 
\begin{equation}
\label{hatdelta}
 \delta^n_m:  C_m^n(V, W)
\to  C_{m-1}^{n+1}(V, W). 
\end{equation} 
for the bicomplex $(C_m^n(V, W), \delta^n_m)$ has the form: 
\begin{equation}
\label{marsha}
\delta^n_m(\Phi)
=E^{(1)}_W\circ_2 \Phi
+\sum_{i=1}^n(-1)^i \Phi\circ_i E^{(2)}_{V; \one}
 +(-1)^{n+1}
\sigma_{n+1, 1, \dots, n}(E^{(1)}_W\circ_2 \Phi),  
\end{equation}
where $\circ_i$ was defined in Section \ref{valued}.    
Explicitly, 
for $(v_1, \dots, v_{n+1})\in V$, $w'\in W'$ 
and $(z_1, \dots, z_{n+1})\in F_{n+1}\C$, 
\begin{eqnarray*}
\lefteqn{\langle w', (( \delta^n_m(\Phi))(v_1 \otimes \cdots\otimes v_{n+1}))
(z_1, \dots, z_{n+1})\rangle}
\nn
&&=R(\langle w', Y_W(v_1, z_1)(\Phi(v_2 \otimes \cdots\otimes v_{n+1}))
(z_2, \dots, z_{n+1})\rangle)
\nn
&&\quad +\sum_{i=1}^n (-1)^i R(\langle w', 
(\Phi(v_1 \otimes \cdots \otimes v_{i-1} \otimes 
Y_V(v_i, z_i-z_{i+1})v_{i+1}
\nn
&&\quad\quad\quad\quad\quad\quad\quad\quad\quad 
\quad\quad\otimes \cdots \otimes v_{n+1}))
(z_1, \dots, z_{i-1}, z_{i+1}, \dots, z_{n+1})\rangle)
\nn
&&\quad + (-1)^{n+1}R(\langle w', Y_W(v_{n+1}, z_{n+1})
(\Phi(v_1 \otimes \cdots \otimes v_n))(z_1, \dots, z_n)\rangle).
\end{eqnarray*}

The following particular case allows to introduce another short bicomplex.  
 For $n=2$, there exists a subspace 
of $C_0^2(V, W)$ 
containing $C_m^2(V, W)$ for all $m\in \Z_+$ such that 
$\delta^2_m$ remains defined on this subspace. 
Let $C_{\frac{1}{2}}^2(V, W)$ be the subspace of $C_0^2(V, W)$ 
consisting of elements $\Phi$ such that for $(v_1$, $v_2$, $v_3) \in V$, $w'\in W'$, 
\begin{eqnarray*}
&& \sum_{r\in \C}\big(\langle w', E^{(1)}_{W}(v_1;
P_r((\Phi(v_2 \otimes v_3))(z_2 -\zeta, z_3-\zeta)))(z_1, \zeta)\rangle
\nn
&&\quad +\langle w', (\Phi(v_1\otimes P_r((E^{(2)}_V(v_2 \otimes v_3; \one)) 
(z_2-\zeta, z_3-\zeta)))) 
(z_1, \zeta)\rangle\big), 
\end{eqnarray*}
\begin{eqnarray*}
\lefteqn{\sum_{r\in \C}\big(\langle w', 
(\Phi(P_r((E^{(2)}_V(v_1\otimes v_2; \one))(z_1 -\zeta, z_2-\zeta))
\otimes v_3))
(\zeta, z_3)\rangle}\nn
&&\quad +\langle w', 
E^{W; (1)}_{WV}(P_r((\Phi(v_1 \otimes v_2))(z_1-\zeta, z_2-\zeta));
v_3))(\zeta, z_3)\rangle\big), 
\end{eqnarray*}
are absolutely convergent in the regions 
\[
|z_1-\zeta|>|z_2-\zeta|,
\]
\[
 |z_2-\zeta|>0, 
\]
\[
|\zeta-z_3|>|z_1-\zeta|,
\]
\[
 |z_2-\zeta|>0,
\]
 respectively, 
and can be analytically extended to 
rational functions in $z_1$ and $z_2$ with the only possible poles at
$z_1, z_2=0$ and $z_1=z_2$.
It is clear that 
\[
C_m^2(V, W)\subset C_{\frac{1}{2}}^2(V, W), 
\]
for $m\in \Z_+$. 
 The coboundary operator 
\begin{equation}
\label{halfdelta}
\delta^2_{\frac{1}{2}}: C_{\frac{1}{2}}^2(V, W)
\to C_0^3(V, W),
\end{equation}
is defined by 
\begin{eqnarray}
\label{halfdelta1}
&& \delta^2_{\frac{1}{2}}(\Phi) 
= E^{(1)}_W \circ_2 \Phi 
+ \sum\limits_{i=1}^2 (-1)^i  E^{(2)}_{V, \one}  \circ_i \Phi
+ E^{W; (1)}_{WV}\circ_2\Phi, 
\end{eqnarray}
\begin{eqnarray*}
\nn
&&\langle w', (({\delta}^2_{\frac{1}{2}}(\Phi))
(v_1\otimes v_2 \otimes v_3))(z_1, z_2, z_3)\rangle
\nn
 &&
=R(\langle w', (E^{(1)}_W(v_1;
\Phi(v_2 \otimes v_3))(z_1, z_2, z_3)\rangle
\nn
&&
\quad \quad 
+\langle w', (\Phi(v_1 \otimes E^{(2)}_V(v_2 \otimes v_3; \one))) 
(z_1, z_2, z_3)\rangle)
\nn
 && \quad 
-R(\langle w', 
(\Phi(E^{(2)}_V(v_1 \otimes v_2; \one))
\otimes v_3))(z_1, z_2, z_3)\rangle
\nn
 &&\quad \quad 
+\langle w', 
(E^{W; (1)}_{WV} (\Phi(v_1 \otimes v_2); v_3))
(z_1, z_2, z_3)\rangle)
\end{eqnarray*}
for $w'\in W'$,
$\Phi\in C_{\frac{1}{2}}^2(V, W)$,
$(v_1$, $v_2$, $v_3) \in V$ and $(z_1, z_2, z_3)\in F_3\C$. 

 In \cite{Huang} one has 
\begin{proposition}
\label{delta-square} 
For $n\in \N$ and $m\in \left\{Z_+ +1\right\}$, 
\begin{equation}
\label{bigcomplex}
0\longrightarrow C_m^{0}(V, W) 
\stackrel{\delta_m^0}{\longrightarrow}
C^1_{m-1}(V, W) 
\stackrel{\delta^1_{m-1}}{\longrightarrow} \cdots 
\stackrel{\delta_1^{m-1}}{\longrightarrow} C_0^m(V, W)\longrightarrow 0,  
\end{equation}
\begin{equation}
\label{shortseq}
0\longrightarrow C_3^0(V, W)
\stackrel{\delta_3^0}{\longrightarrow}
C_2^1(V, W)
\stackrel{\delta_{2}^{1}}{\longrightarrow}C_{\frac{1}{2}}^{2}(V, W)
\stackrel{\delta_{\frac{1}{2}}^{2}}{\longrightarrow}
C_0^3(V, W)\longrightarrow 0. 
\end{equation}
\begin{equation}
\label{berta1}
 \delta^{n+1}_{m-1}\circ \delta^n_m=0,
\end{equation} 
\begin{equation}
\label{berta2}
\delta^2_{\frac{1}{2}}\circ {\delta}^1_2=0.
\end{equation}
are chain-cochain bicomplexes 
 with respect to  
the coboundary operators \eqref{marsha} and \eqref{halfdelta1}.  
\end{proposition}
In what follows we omit $(V, W)$ from notations of complexes. 
\section{The multiplication of elements of the chain complex spaces} 
\label{multiplication}
In this Section  
we introduce the definition of the simplest variant of converging multiplication $*_l$ 
 of $l$ elements of bicomplex spaces  
with the image in another bicomplex space coherent with respect 
 to the original differential \eqref{hatdelta}, 
and satisfying the symmetry \eqref{shushu},  
$L_V(0)$-conjugation \eqref{loconj}, and $L_V(-1)$-derivative \eqref{lder1}--\eqref{lder2} properties 
described in Section \ref{valued}. 

Let $k_0=0$. 
We consider $l$ sets of grading-restricted vertex algebra elements 
$(v_{\sum_{j=1}^i k_{j-1}+1}$, $\ldots$, $v_{\sum_{j=1}^i k_{j-1}+k_j})$ 
and corresponding formal parameters  
$(x_{\sum_{j=1}^i k_{j-1}+1}$, $\ldots$, $x_{\sum_{j=1}^i k_{j-1}+k_j})$
for elements 
\[
\Phi_i(v_{\sum_{j=1}^i k_{j-1}+1} \otimes \ldots \otimes  
v_{\sum_{j=1}^i k_{j-1}+k_j})
(x_{\sum_{j=1}^i k_{j-1}+1}, \ldots, x_{\sum_{j=1}^i k_{j-1}+k_j}) \in C^{k_i}_{m_i}, 
\]    
of the bicomplexes \eqref{bigcomplex} and \eqref{shortseq}. 
It is assumed that $(x_{\sum_{j=1}^i k_{j-1}+1}, \ldots, x_{\sum_{j=1}^i k_{j-1}+k_j})$
belong to $F_{k_i}\C$.
The same property we require from the set of formal parameter for a multiplication of 
 elements of $l$ spaces $C^{k_i}_{m_i}$, $1 \le i \le l$.  
Therefore, according to the 
 definition of the configuration space $F_n\C$, 
for each group of coinciding formal parameters among $l$ sets
 $(x_{\sum_{j=1}^i k_{j-1}+1}, \ldots, x_{\sum_{j=1}^i k_{j-1}+k_j})$
we keep only one parameter at the first appearance. 
Since the elements $\Phi_i$ are rational functions given my matrix elements,  
 exclusions of formal parameters leads to exclusions of corresponding monoms  
for $x_{i_j}=x_{i'_{j'}}$, $1 \le i_j \le k_i$, $1 \le i'_{j'} \le k_{i'}$,  
 $1 \le j\ne j' \le l$.  

Let $r_i$, $1 \le i \le l$ be the number of dropped formal parameters in each of $l$ sets. 
The same procedure we do for $l$ sets of vertex operators composable with $\Phi_i$, and 
 we drop $t_i$ vertex operators.  
Then resulting sets contain $k_i-r_i$ parameters and 
$\Phi_i$ belong to $C^{k_i-r_i}_{m_i-t_i}$ composable with $m_i-t_i$ vertex operators.  
The operation of exclusion will be denotes by $\widehat{f}$. 

 Define $n_i=\sum\limits_{j=1}^i (k_{j-1}-r_{j-1})$.  
Let us put $k=\sum\limits_{i=1}^l (k_i-r_i)$.  
We obtain $l$ sets of formal parameters 
 $(x_{n_i+1}, \ldots, x_{n_i+k_i}) \in F_{k_i}\C$. 
 Then, we define   
\begin{eqnarray}
\label{zsto}
&& (z_1, \ldots, z_k)= 
(x_{n_1+1},  \ldots, x_{n_1+k_1-r_1}, \ldots, x_{n_l+1}, \ldots, x_{n_l+k_l-r_l}), 
\end{eqnarray} 
\begin{eqnarray}
\label{zsto1}
({\rm v}_1, \ldots, {\rm v}_k)= 
(v_{n_1+1},  \ldots, v_{n_1+k_1-r_1}, \ldots, v_{n_l+1}, \ldots, v_{n_l+k_l-r_l}). 
\end{eqnarray} 
In what follows we will use \eqref{zsto}, \eqref{zsto1} for notating formal parameters as well as 
vertex algebra elements of the multiplication. 
The simplest possible multiplication of elements 
of several $C^{k_i}_{m_i}$-spaces, $1 \le i \le l$, 
 is defined by  
multiplications of matrix elements of the form \eqref{pochta} summed over a $V_{(l)}$-basis for $l\in \Z$.  
\begin{definition}
\label{paraska}
For 
$\Phi_i(v_{n_i+1} \otimes \ldots \otimes v_{n_i+k_i-r_i})  
(x_{n_i+1}, \ldots,  x_{n_i+k_i-r_i})  \in  C^{k_i-r_i}_{m_i-t_i}$,    
and non-vanishing $\zeta_{1, i}$, $\zeta_{2, i} \in \C$, such that  
 $\zeta_{1, i} \zeta_{2, i} = \epsilon$,    
$1 \le i \le l$, 
 we define the multiplication 
\begin{eqnarray}
\label{gendef}
 *_l: \times_{i=1}^l \Phi_i(v_{n_i+1} \otimes \ldots \otimes v_{n_i+k_i}) (x_{n_i+1}, \ldots,  x_{n_i+k_i})   
  \qquad \qquad \qquad \qquad  \qquad \qquad \qquad   && 
\nn
\;  \mapsto 
\widehat{R} \; \Phi\left( v_1 \otimes \ldots \otimes v_{k_1}  
\otimes v_{k_1+1} \otimes \ldots 
\otimes 
 v_{n_l+k_l} \right) 
\left(x_1, \ldots, x_{k_1}, x_{k_1+1}, \ldots, x_{n_l+k_l}; \epsilon, \zeta_{a, i} \right), \qquad &&   
\end{eqnarray}
\begin{eqnarray}
\label{Z2n_pt_epsss}
&& \langle w',  \widehat{R} \; \Phi  
\left( v_1 \otimes \ldots \otimes v_{k_1} \otimes v_{k_1+1} \otimes \ldots 
\otimes 
 v_{n_l+k_l} \right) 
\left(x_1, \ldots, x_{k_1}, x_{k_1+1}, \ldots, x_{n_l+k_l} ; \epsilon, \zeta_{a, i}\right)    
\rangle 
\nn 
&&   =  
R   \sum_{n \in \Z} \sum_{u_n\in V_{(n)} }   \epsilon^n 
 \langle w', \prod\limits_{i=1}^l Y^W_{WV} \left(   
\Phi_i({\rm v}_{n_i+1} \otimes \ldots \otimes {\rm v}_{n_i+k_1-r_1} \otimes u_n) \right. 
\nn
&& \left. \qquad \qquad \qquad \qquad \qquad \qquad \qquad 
(z_{n_i+1}, \ldots, z_{n_i+k_i-r_i}, \zeta_{1,l})    
, \zeta_{2, l} \right)\; \overline{u}_n \rangle,     
\end{eqnarray}
 for only those converging matrix elements 
\begin{eqnarray}
\label{element0}
 \mathcal M^{(i)}(w'; z_{n_i+1}, \ldots, z_{n_i+k_i-r_i}; u_n, \zeta_{a, i}) 
\qquad \qquad \qquad \qquad \qquad \qquad \qquad \qquad \qquad 
 &&
\nn
  =\langle w', Y^W_{WV} \left(  
\Phi_i({\rm v}_{n_i+1} \otimes \ldots \otimes {\rm v}_{n_i+k_i-r_i} \otimes u_n)    
(z_{n_i+1}, \ldots, z_{n_i+k_i-r_i}, \zeta_{1,i})   
, \zeta_{2, i} \right)\; \overline{u}_n \rangle,  &&
\end{eqnarray}
such that \eqref{Z2n_pt_epsss} is absolutely convergent as a series in $\epsilon$.  
\end{definition} 
The multiplication is parametrized by non-vanishing 
$\zeta_{1, i}$, $\zeta_{2, i}  \in \C$, $1 \le i \le l$.  
The sum 
is taken over any $V_{(n)}$-basis $\{u_n\}$,   
where $\overline{u}_n$ is the dual of $u_n$ with respect to a non-degenerate bilinear form 
 $\langle .\ , . \rangle$, 
 \eqref{formain}  
over $V$. 
The operation $\widehat{R}$ eliminates $r=\sum\limits_{i=1}^lr_i$ coinciding formal parameters 
from $\Phi$ in the sets $(x_{n_i+1}, \ldots, x_{n_i+k_i})$ 
and excludes all monomials $(x_{i_j} - x_{j_j})$, $1 \le i_j \le k_i$, 
$1 \le j \le l$ in   
\eqref{Z2n_pt_epsss}. 
By the standard reasoning \cite{FHL, Zhu}, 
 \eqref{Z2n_pt_epsss} does not depend on the choice of a basis of $u_n \in V_{(n)}$, $n \in \Z$.   
The form 
of the multiplication defined above is natural in terms 
of the theory of characters for vertex operator algebras 
\cite{TUY, FMS, Zhu}.   

In the case of converting 
\eqref{Z2n_pt_epsss} 
we introduce the action of the partial derivative  
$\partial_p=\partial_{z_p}={\partial}/{\partial_{z_p}}$, 
with respect to the $p$-th entry of 
 $(z_1,  \ldots,  z_k)$, 
 $1\le p \le k$, 
 as follows 
\begin{eqnarray}
\label{Z2n_pt_eps1qdef}
 & &
\langle w', \partial_p  
\Phi
({\rm v}_1 \otimes  \ldots \otimes {\rm v}_k)  
(z_1,  \ldots,  z_k; \epsilon, \zeta_{a, i}) 
\rangle 
\nn
 & &  = \sum_{n \in \mathbb{Z}
} \epsilon^n \sum_{u_n\in V_{(n)}}   
\langle w', \partial^{\delta_{p,i}}_{z_i}  \prod_{i=1}^l 
Y^W_{WV}\left( 
   \Phi_i( {\rm v}_{n_i+1} \otimes   \ldots \otimes {\rm v}_{n_i+k_i-r_i} \otimes u_n) 
\right. 
\nn
&&
\left. 
\qquad (z_{n_i+1}, \ldots, z_{n_i+k_i-r_i}, \zeta_{1, i})  
, \zeta_{2, i} \right)\; \overline{u}_n \rangle.    
\end{eqnarray}
Similarly, we define the 
action of an element $\sigma \in S_k$ on the  
for convering \eqref{Z2n_pt_epsss}   
as
\begin{eqnarray}
\label{Z2n_pt_epsss1}
&& \langle w',  \sigma(\widehat{R}\;\Phi) 
({\rm v}_1 \otimes \ldots 
\otimes 
{\rm v}_k) 
(z_1,  \ldots,  z_k ; \epsilon, \zeta_{a, i})    
) \rangle 
\nn
&&
\qquad =\langle w',  \Phi 
({\rm v}_{\sigma(1)} \otimes \ldots  
 \otimes 
 {\rm v}_{\sigma(k)}) 
(
z_{\sigma(1)}, 
  \ldots, 
z_{\sigma(k)}; \epsilon, \zeta_{a, i} 
) \rangle 
\nn 
& & \qquad =  
\sum\limits_{n \in \Z} 
\sum_{u_n\in V_{(n)} }  
 \langle w', \prod_{i=1}^l Y^W_{WV}\left(   
\Phi_i 
({\rm v}_{\sigma(n_i+1)} \otimes \ldots \otimes {\rm v}_{\sigma(n_i+k_i-r_i)} \otimes u_n)   
\right. 
\nn
&&
\left.
\qquad \qquad (z_{\sigma(n_i+1)}, \ldots,  z_{\sigma(n_i+k_i-r_i)}, \zeta_{1, i})  
, \zeta_{2, i}\right)\; \overline{u}_n \rangle.   
\end{eqnarray}
Let $t$ be the number of common vertex operators the mappings  
$\Phi_i({\rm v}_{n_i+1} \otimes \ldots {\rm v}_{n_i+k_i-r_i})$     
$(z_{n_i+1}$,   $\ldots$, $z_{n_i+k_i-r_i})   
\in C^{k_i-r_i}_{m_i-t_i}$, 
are composable with.  
Let us put $m=\sum\limits_{i=1}^l m_i-t_i$. 
 We then have   
\begin{proposition}
\label{tolsto}
For $\Phi_i ({\rm v}_{n_i+1} \otimes  
\ldots \otimes {\rm v}_{n_i+k_i-r_i}) (z_{n_i+1}, \ldots, z_{n_i+k_i-r_i})   
\in C_{m_i-t_i}^{k_i-r_i}$ 
and converging \eqref{element0},   
the multiplication  
$\widehat{R} \; \Phi 
\left({\rm v}_1 \otimes \ldots  \otimes  {\rm v}_k \right)   
\left(z_1, \ldots,  z_k; \epsilon, \zeta_{a, i}\right)$,   
 \eqref{Z2n_pt_epsss} 
belongs to the space $C^k_m$, i.e.,    
\begin{equation}
\label{toporno}
*_l: \times_{i=1}^l C^{k_i-r_i}_{m_i-t_i} \to  C_m^k,  
\end{equation}
with non-vanishing $\zeta_{1, i}$, $\zeta_{2, i}$, such that 
$\zeta_{1, i} \zeta_{2, i} = \epsilon$,   
$1 \le i \le l$. 
\end{proposition}
In what follows, let us skip $\epsilon$ and $\zeta_{a, i}$, $a=1$, $2$, $1\le i \le l$ from 
the notation of $\Phi\left({\rm v}_1 \otimes \ldots  \otimes  {\rm v}_k \right)   
\left(z_1, \ldots,  z_k; \epsilon, \zeta_{a, i} \right)$. 
\begin{proof} 
We show that
 \eqref{Z2n_pt_epsss} converges to a    
$\overline{W}$-valued rational function defined on the configuration space $F\C_k$,   
for formal variables 
with only extra possible poles at $z_{i_q}=z_{j_q}$, $....1 \le i_q \le k$, $1 \le j \le k$,  
satisfies \eqref{shushu}, $L_V(0)$-symmetry \eqref{loconj}, 
and $L_V(-1)$-derivative \eqref{lder1}--\eqref{lder2} 
conditions, 
and composable with $m$ vertex operators.  
In order to prove convergence of a multiplication 
of elements of the spaces $C_{m_i-t_i}^{k_i-r_i}$, $1\le i \le l$, 
we use a geometrical interpretation \cite{H2, Y} of \eqref{Z2n_pt_epsss}
in terms of sewing Riemann surfaces with marked points. 
Recall that a $C_m^k$-space 
is defined by means of matrix elements \cite{FHL} of the form \eqref{pochta},
 and satisfying $L_V(0)$-conjugation,  
$L_V(-1)$-derivative conditions, \eqref{shushu}, and composable with $m$ vertex operators.  
For a vertex algebra $V$, and it module $W$, satisfying certain extra conditions \cite{TZ},  
one associates elements of a space $C_m^k$ with the data on the Riemann sphere.  
In particular, formal parameters of $C_m^k$-elements and vertex operators they are composable to, are 
 identified with local coordinates of marked points on a sphere.  

For a pair of spaces $C_{m_i}^{k_i}$, $1 \le i \le l$, 
 we consider data on $l$ Riemann spheres. 
Two extra points are chosen for centers of annuli used in order 
to sew spheres \cite{Y, H2} to obtain another sphere.     
The resulting multiplication \eqref{Z2n_pt_epsss}  
represents a sum of multiplications of matrix elements originated from 
$l$ original Riemann spheres.   
Two sets of non-vanishing complex parameters $\zeta_{1, i}$, $\zeta_{2, i}$, $1 \le i \le l$,  
 of \eqref{Z2n_pt_epsss} are identified with coordinates  
on annuli. 
After identification of annuli 
$r$ coinciding coordinates may occur. 
This takes into account case of coinciding formal parameters.

For $l$ sets of vertex algebra elements $(v_{n_i+1}, \ldots, v_{n_i+k_i})$,   
and $l$ sets of  
formal complex parameters 
$(z_{n_i+1}, \ldots, z_{n_i+k_i})$,  
 the formal parameters are identified with the local coordinates of 
$k_i-r_i$ points on the $i$-copy of the Riemann sphere $\widehat{\Sigma}^{(i)}$, 
 with excised annuluses ${\mathcal A}_{a, i}$, $a=1$, $2$.    
 We assume that the sewing parameter condition \cite{Y} is 
$\zeta_{1, i} \zeta_{2, i} = \epsilon$,  
$1 \le i \le l$. 
 In Riemann surface $\epsilon$-sewing formulation, 
the complex parameters $\zeta_{a, i}$, $a=1$, $2$ are coordinates inside identified annuluses,  
and 
$0< |\zeta_{a, i}|\le r_{a, i}$.  
Then, we obtain 
\begin{eqnarray}
\nonumber
& & 
\langle w',  
 R \Phi({\rm v}_1 \otimes \ldots \otimes {\rm v}_k) (z_1, \ldots, z_k; \epsilon, \zeta_{a, i})\rangle     
\\
\label{Z2n_pt_eps1q11}
 & &  \quad = \sum\limits_{n\in \Z} \epsilon^n \sum_{u \in V_{(n)}}    
\prod_{i=1}^l 
\langle w', Y^W_{WV}\left(  
\Phi_i({\rm v}_{n_i+1} \otimes \ldots \otimes {\rm v}_{n_i+k_i-r_i} \otimes u_n)
\right. 
\nn
&&
\qquad \qquad \qquad \qquad \qquad \qquad \qquad \qquad \left. 
 (z_{n_i+1}, \ldots, z_{n_i+k_i-r_i}, \zeta_{1, i} \right)\; \overline{u}_n  \rangle,   
\end{eqnarray}
\begin{eqnarray*}
&&
=  \sum_{n\in \mathbb{Z} } \epsilon^n \sum_{u\in V_{(n)}}   
\prod_{i=1}^l
\langle w', e^{\zeta_{2, i} \; L_W{(-1)}} \; Y_W \left( \overline{u}, -\zeta_{2, i}\right) \; 
\Phi_i({\rm v}_{n_i+1} \otimes \ldots \otimes {\rm v}_{n_i+k_i-r_i} \otimes u_n) 
\nn
&&
\qquad \qquad \qquad \qquad \qquad \qquad  \qquad \qquad \qquad \qquad  \qquad \qquad 
(z_{n_i+1},  \ldots, z_{n_i+k_i-r_i}, \zeta_{1, i})  \rangle.  
\end{eqnarray*}
Recall that we assume that 
the matrix elements \eqref{element0} 
are absolutely convergent in powers of 
$\epsilon$ with some radia of convergence $R_{a, i}\le 
r_{a, i}$, with  
$0 < |\zeta_{a, i}|\le R_{a, i}$, $1 \le i \le l$. 
The dependence of \eqref{element0} on $\epsilon$ is expressed via the relation 
$\epsilon=\zeta_{1, i} \zeta_{2, i}$.  
Let us rewrite the multiplication \eqref{Z2n_pt_eps1q11} as 
\begin{eqnarray}
\label{pihva}
&& \langle w',  
 R\Phi({\rm v}_1 \otimes \ldots \otimes {\rm v}_k)(z_1, \ldots, z_k; \epsilon, \zeta_{a, i}) \rangle      
\nn
&& 
\qquad = 
  \sum\limits_{n \in Z} \epsilon^n \left(  \prod_{i=1}^l \langle w',   
 \Phi({\rm v}_1 \otimes \ldots \otimes {\rm v}_k \otimes u_n)
 (z_1, \ldots, z_k, \zeta_{1, i})  \rangle \right)_n  
\nn
&&
=  \sum_{n \in \mathbb{Z}} 
 \sum_{u_n\in V_{(n)}} \sum\limits_{q \in \C}
 \epsilon^{l-q-1} \; \prod_{i=1}^l
 \left( {\mathcal M}^{(i)}(z_{n_i+1}, \ldots, z_{n_i+k_i-r_i}; u_n, \zeta_{a, i}) \right)_q, 
\end{eqnarray}
  as a formal series in $\epsilon$    
for $0 < |\zeta_{a, i}|\leq R_{a, i}$, where  
$|\epsilon |\leq r$ for $r<r_{a, i} r_{a, i}$.  
For $R_i={\rm max}\left\{R_{a, i}\right\}$, $a=1$, $2$, 
 we apply 
 Cauchy's inequality to the 
 coefficient forms \eqref{element0} 
 to find 
\begin{equation}
\label{Cauchya}
 \left| \left({\mathcal M}^{(i)}  
(z_{n_i+1}, \ldots, z_{n_i+k_i-r_i}; u_n, \zeta_{a, i}) \right)_q \right| \leq {M_i} {R_i^{-q}},     
\end{equation}
with 
\[
M_a=\sup_{ 0 < \left| \zeta_a \right| \leq R_a,\; \left|\epsilon \right| \leq r} 
\left| \mathcal M^{(i)}(z_{n_i+1}, \ldots, z_{n_i+k_i-_i}; u_n, \zeta_a ) \right|.    
\]
Using \eqref{Cauchya} we obtain for \eqref{pihva}, 
$M={\rm min} \left\{M_i \right\}$,  and $M={\rm max} \left\{R_i \right\}$
 $1 \le i \le l$,  
\begin{eqnarray}
\label{borona} 
&&
\left| 
 \left(\langle w',   
 R\Phi ({\rm v}_1 \otimes \ldots \otimes {\rm v}_k) (z_1, \ldots, z_k; \epsilon, \zeta_{a, i})  
\rangle \right)_n\right|
\nn
&&
\qquad \le   
 \prod_{i=1}^l \left|\left(\mathcal M^{(i)} (z_{n_i+1}, \ldots, z_{n_i+k_i-r_i}; u_n, \zeta_1) \right)_n\right|  
\nn
&&
\qquad  \le  
\prod_{i=1}^lM_i R_i^{-n+q+1} \le \left(MR\right)^{l(-n+q+1)}.    
\end{eqnarray}
 We see that \eqref{Z2n_pt_eps1q11} is absolute convergent 
 as a formal series in $\epsilon$   
 and defined 
for $0 < |\zeta_a|\leq r_a$,  
$|\epsilon |\leq r$ for $r<r_1r_2$, 
 with extra poles 
only at $z_{i_j}=z_{i'_{j'}}$, $1\le i_j \le k_i-r_i$, $1\le {i'}_{j'} \le k_{i'}-r_{i'}$,  
$1\le j \ne j' \le l$.  
Elements $\Phi_i \in C^{k_i}_{m_i}$ 
 are defined on the configuration spaces $F\C_{k_i}$.     
The extension is from $l$ original Riemann spheres to the sphere formed by the sewing procedure. 
The construction of \eqref{Z2n_pt_epsss} gives the $\overline{W}$-valued 
rational function on the configuration space $F\C_k$.   
By construction, the element $R\Phi$  
is analytically extendable to 
a $\overline{W}$-valued function defined on the configuration space  
 $F\C_k=\left\{(z_1,  ...,  z_k):  z_i \neq z_j, i \neq j \right\}$.     
Due to 
the construction 
of the multiplication \eqref{Z2n_pt_epsss}, 
 for every $u_n \in V_{(n)}$, each summand in \eqref{Z2n_pt_epsss} defines a rational function 
  on the configuration space $F\C_k$.    
By direct substitution we prove that 
 the multiplication \eqref{Z2n_pt_epsss} satisfies the $L_V(-1)$-derivative 
\eqref{lder1}--\eqref{lder2} and $L_V(0)$-conjugation \eqref{loconj} properties. 
Using the definition of the action of an element 
$\sigma \in S_k$ on the multiplication \eqref{Z2n_pt_epsss},  
we prove \eqref{shushu} for \eqref{Z2n_pt_epsss}. 

Recall that 
$\Phi_i({\rm v}_{n_i+ 1} \otimes \ldots \otimes {\rm v}_{n_i+ k_i})  
(z_{n_i + 1}, \ldots,  z_{n_i+k_i})$, 
are composable with $m_i$ vertex operators 
with $l$ sets of vertex algebra elements 
$(v'_{p_i+1}, \ldots, v'_{p_i+m_i})$, 
and $l$ sets of corresponding formal parameters 
$(z'_{p_i+1}, \ldots, z'_{p_i+m_i})$ for $p_i=\sum\limits_{j=1}^i (m_{j-1}-t_{j-1})$.  
We denote by $({\rm v}'_{p_i+1}, \ldots, {\rm v}'_{p_i+m_i-t_i})$ $l$ sets of vertex operators 
after dropping $t$ coinciding vertex operators. 
Next, we show that 
 \eqref{Z2n_pt_epsss}  
is composable with $m$ vertex operators. 
 We redefine the notations for the 
set
\begin{eqnarray*}
  &&(v''_1, \ldots, v''_k; v''_{k+1}, \ldots, v''_{k+m})  
\nn
&&
=({\rm v}_1, \ldots, {\rm v}_{k_1-r_1}; {\rm v}'_1, \ldots, {\rm v}'_{m_1-t_1}; \ldots; 
{\rm v}_{n_l+1}, \ldots, {\rm v}_{n_l+k_l-r_l}; 
 {\rm v}'_{p_l+1}, \ldots, {\rm v}'_{p_l+m_l-t_l}), 
\end{eqnarray*}
 of vertex algebra $V$ elements and similarly 
for formal parameters $z''_j$, $1 \le j \le k+m$.  
Let us consider the first condition of 
composability for the multiplication 
\eqref{Z2n_pt_epsss}   
with a number of vertex operators. 
Introduce $l''_1, \dots, l''_k \in \Z_+$,     
 such that $\sum_{j=1}^k l''_j = k+m$.    
Define  
 \begin{eqnarray}
\label{psiinew}
\varphi''_i
&
=
&
\varepsilon^{(l''_{i''})}_V 
(v''_{k''_1}\otimes  
 \ldots \otimes
v''_{k''_{i''}}
 ; \one)
(z_{k''_1}- \zeta''_{i''},  
 \ldots, 
 z_{k''_{i''}}- \zeta''_{i''}   
),        
\end{eqnarray}
where for $l''_0=1$, 
\begin{eqnarray}
\label{ki}
 {k''_1}=\sum_{i=0}^{i''-1} l''_j, 
 \quad  \ldots, \quad  {k''_{i''}}={l''_1 +\ldots +l''_{i''-1}+l''_{i''}},   
\end{eqnarray} 
for $1 \le i'' \le k$,  
and we take 
\[
(\zeta''_1, \ldots, \zeta''_k)= (\zeta_1, \ldots, \zeta_{k_1}; \zeta'_1, \ldots, \zeta'_k).  
\]  
Then we consider 
\begin{eqnarray}
\label{Inmdvadva}
  \mathfrak I^k_m(\widehat R\; \Phi)=    
\sum_{r''_1, \dots, r''_k\in \Z} 
 \langle w',  
\widehat R\; \Phi (P_{r''_1}\varphi''_1 \otimes 
 \ldots \otimes 
P_{r''_k} \varphi''_k)  
( \zeta''_1, 
 \ldots, 
 \zeta''_k; \epsilon, \zeta_{a, i}) 
\rangle, &&
\end{eqnarray} 
and prove it is absolutely convergent with some conditions. 
 The condition   
\begin{eqnarray}
\label{granizy1000000}
&&
|z_{l''_1+\ldots +l''_{i-1}+p''} -\zeta''_i| 
+ |z_{l''_1+\ldots +l''_{j-1}+q''} -\zeta''_i|< |\zeta''_i -\zeta''_j|,  
\end{eqnarray} 
of absolute convergence for \eqref{Inmdvadva} for $1'\le i''\ne j''  \le k$, 
and for $1 \le p'' \le l''_i$ and $1 \le q'' \le l''_j$, follows from corresponding conditions
for $\Phi_i$, $1 \le i \le l$.  
 We obtain 
\begin{eqnarray*}
 && \left|\mathfrak I^k_m(\widehat R \; \Phi)\right|  
\le \prod\limits_{i=1}^l \left|\mathfrak I^{k_i-r_i}_{m_i-t_i}(\Phi_i)\right|.    
\end{eqnarray*} 
Thus, we infer that \eqref{Inmdvadva}
is absolutely convergent. 
Recall that 
the maximal orders of possible poles of \eqref{Inmdvadva} 
are $N^{k_i-r_i}_{m_i-t_i}(v_i, v_j)$,  
 at $z_{i_j}=z_{i'_{j'}}$, $1 \le i_j \le k_i$, $1 \le j \le l$,  
$z_{i'}=z_{j'}$, $k_1+1 \le i'$, $j' \le k$.  
From the last expression we deduce that 
there exist 
  positive integers $N^k_m(v''_{i''}, v''_{j''})$  
%
 for $1 \le i\ne j \le k_1$, $1 \le i'\ne j\le k_2$,  
depending only on $v''_{i''}$ and $v''_{j''}$ for  $1\le  i''\ne j'' \le k$, 
such that 
 the series \eqref{Inmdvadva} 
can be analytically extended to a
rational function
in $(z_1, \dots, z_{k_1}; z_{k_1+1}, \ldots, z_k)$,      
 independent of $(\zeta''_1, \dots, \zeta''_k)$,     
with extra 
 possible poles at  
 and $x_i=y_j$,  of order less than or equal to 
$N^k_m(v''_{i''}, v''_{j''})$, for $1 \le i''\ne j'' \le k_2$. 

Let us proceed with the second condition of composability. 
For the multiplication \eqref{Z2n_pt_epsss} we obtain 
  $(v''_1, \dots, v''_{k+m})\in V$,   
and 
$(z_1, \ldots $, $z_{k+m} )\in \C$,   
we find 
positive integers  
$N^k_m(v'_i, v'_j)$,   
 depending only on $v'_i$ and 
$v''_j$, for $i''$, $1 \le j'' \le k$, $i''\ne j''$, such that for arbitrary $w'\in W'$. 
Under conditions
\begin{eqnarray}
\label{granizy2}
z_{i''}\ne z_{j''}, \quad i''\ne j'', \quad 
\nn
|z_{i''}|>|z_{k'''}|>0, 
\end{eqnarray}
 for $1 \le i'' \le m$, and $m+1 \le k''' \le m+k$, 
let us introduce 
\begin{eqnarray}
\label{perda}
&&
 \mathfrak J^k_m(\widehat R \; \Phi) = \sum_{q\in \C}  
\langle w', E^{(m)}_W \Big(  
v''_1 \otimes \ldots \otimes 
v''_m;   
 P_q( \widehat R \; \Phi 
( v''_{m+1} \otimes  
\ldots \otimes 
v''_{m+k}  
)
\nn
&&
\qquad \qquad \qquad \qquad \qquad \qquad   ( z_{m+1},  
\ldots, 
z_{m+k}; \epsilon, \zeta_{a, i})  
)\Big) 
(z_1,  \ldots, 
 z_m)\rangle.  
\end{eqnarray}
We then obtain 
\begin{eqnarray*}
&&
|\mathfrak J^k_m(\widehat R \; \Phi) |  
\le \prod\limits_{i=1}^l \left| 
\mathfrak J^{k_i-r_i}_{m_i-t_i}(\Phi_i) \right|,    
\end{eqnarray*}
where we have used 
the invariance of \eqref{Z2n_pt_epsss} with respect to 
$\sigma \in S_{k+m}$.  
 $\mathcal J^{k_1}_{m_1}(\Phi_1)$ and $\mathcal J^n_{m'}(\Phi_2)$
 in the last expression 
are absolute convergent.  
Thus, we infer that  $\mathfrak J^k_m(\widehat R \; \Phi)$  
is absolutely convergent, and 
the sum \eqref{Inmdvadva} 
 is analytically extendable to a  rational function  
in $(z_1, \dots, z_{k+m})$ with the only possible poles at 
$x_i=x_j$, $1 \le i$, $j \le k_1$, $y_{i'}=y_{j'}$, $1 \le i'$, $j'\le k$, 
 and extra possible poles    
at $x_i=y_{j'}$, i.e., the only possible poles at 
$z_{i''}=z_{j''}$, of orders less than or equal to 
$N^k_m(v''_{i''}, v''_{j''})$,    
for $i''$, $j''=1, \dots, k'''$, $i''\ne j''$.  
This finishes the proof of the proposition. 
\end{proof}
The multiplication admits the action of 
the differential operator $\delta^k_m$ defined in   
\eqref{marsha} and \eqref{halfdelta1} for  
elements $\Phi_i \in C^{k_i-r_i}_{m_i-t_i}$, $1 \le i \le l$.     
 The co-boundary operators \eqref{marsha} and \eqref{halfdelta1} 
 possesse a variation of Leibniz law with respect to the $*$-multiplication.  
By direct computation one checks 
\begin{proposition}
\label{tosya}
For arbitrary $w' \in W'$, 
  $\Phi_i({\rm v}_{n_i+1} \otimes  
\ldots
 \otimes {\rm v}_{n_i+k_i-r_i})$ 
 $(z_{n_i+1}$,  $\ldots$,  $z_{n_i+k_i-r_i})\in C_{m_i-t_i}^{k_i-r_i}$, $1 \le i \le l$, 
the action of $\delta_m^k$ on  the converging \eqref{Z2n_pt_epsss} is given by   
\begin{eqnarray}
\label{leibniz0}
 && 
\langle w', 
\delta_m^k  (R \Phi({\rm v}_1 \otimes    
 \ldots \otimes {\rm v}_k)  
(z_1,   \ldots,   z_k; \epsilon, \zeta_{a, i}) 
  \rangle \qquad \qquad \qquad \qquad \qquad \qquad 
\qquad \qquad \qquad \qquad  
\qquad \qquad \qquad \qquad \qquad \qquad   
\nn
&&
 \; = 
 \langle w', \sum_{q=1}^l (-1)^{q+1} \prod_{i=1}^l  
\left( 
\delta^{k_i-r_i}_{m_i-t_i} \Phi_i ({\rm v}_{n_i+1} \otimes   
  \ldots \otimes {\rm v}_{n_i+k_i-r_i})   
(z_{n_i+1},  \ldots, z_{n_i+k_i-r_i})  \right) \rangle. 
\nn
&& 
\end{eqnarray}
The multiplication \eqref{Z2n_pt_epsss} extends the chain-cochain 
complex 
structure of Proposition \ref{delta-square}
to all multiplications 
 $*_{i=1}^l C^{k_i-r_i}_{m_i-t_i}$,    
$1 \le i \le l$,
 $m_1$, $m_2 \ge 0$.   
\end{proposition}
\begin{proof}
According to \eqref{marsha} (correspondingly, \eqref{halfdelta1}), the action of 
$\delta_m^k$   
on $R\Phi({\rm v}_1 \otimes \ldots \otimes {\rm v}_k)$  
 $(z_1, \ldots, z_k$; $\epsilon$;       
   $\zeta_{1, i}$, $\zeta_{2, i} )$,  
it is given by  
(we assume, as before, that the vertex operator 
$Y_V ({\rm v}_j, z_j - z_{j+1})$   
does not act on $(u_n, \zeta_{a, i})$)   
\begin{eqnarray}
\label{paros}
 && \langle w', \delta_m^k R \;  
\Phi({\rm v}_1\otimes \ldots \otimes {\rm v}_k) (z_1, \ldots, z_k; \epsilon, \zeta_{a, i}) \rangle    
\nn
 && \qquad \; = \langle w', 
  \sum_{j=1}^k(-1)^j \;    
  R\; \; \Phi ({\rm v}_1 \otimes \ldots  {\rm v}_{j-1} \otimes       
Y_V ({\rm v}_j, z_j - z_{j+1}) {\rm v}_{j+1} 
\nn 
&& \qquad \qquad \qquad \qquad \otimes {\rm v}_{j+2} \otimes    
\ldots \otimes {\rm v}_{k+1}) (z_1, \ldots, \widehat{z}_j T_j(z_{j+1}), 
\ldots, z_{k+1}; \epsilon,  \zeta_{a, i}) 
\nn
 && \quad \; +   
 R \; \left(Y_W \left({\rm v}_1, z_1\right) \; \Phi( {\rm v}_2\otimes   
\ldots \otimes {\rm v}_{k+1})(z_2, \ldots, z_{k+1}; \epsilon, 
 \zeta_{a, i}) \right) 
\nn
 && \quad \; + (-1)^{k+1}      
R\left( Y_W( {\rm v}_{k+1}, z_{k+1}) \; \Phi(
{\rm v}_1\otimes 
 \ldots \otimes {\rm v}_k)(z_1, \ldots, z_k; \epsilon, 
 \zeta_{a, i}) \right)\rangle, 
\end{eqnarray}
where we denote by $T_j(g)$ the insertion of an element $g$ at the $j$-th place. 
Using \eqref{Z2n_pt_eps1q11} we see that the above is equivalent to 
\begin{eqnarray*}
    \sum_{j=1}^k(-1)^j     
R \sum_{u_n\in V_{(n)} \atop n \in \Z}      
\prod_{i=1}^l 
 \epsilon^n  
 \langle w, Y^W_{WV'} \left( 
\Phi_i({\rm v}_{n_i+1}\otimes  \ldots 
\right. \qquad \qquad \qquad \qquad  \qquad \qquad &&
\nn
 \otimes 
\; {\rm v}_{j-1} \otimes Y_V ({\rm v}_j, z_j - z_{j+1})\; {\rm v}_{j+1} \otimes  
 {\rm v}_{j+2} \otimes 
\ldots 
\otimes {\rm v}_{n_i+k_i-r_i+1} \otimes u_n) &&
\nn
\left. 
(z_{n_i+1}, \ldots, \widehat{z}_j T_j(z_{j+1}), \ldots, z_{n_i+k_i-r_i+1}, \zeta_{1, i} ), 
\zeta_{2, i} \right)\; \overline{u}_n \rangle,  &&
\end{eqnarray*}
\begin{eqnarray*}
  +  
R \sum_{u_n\in V_{(n)} \atop n \in \Z}     
\prod_{i=1}^l
\sum_{n \in \Z } \epsilon^n  \langle w,  \left(Y_W ({\rm v}_1, z_1) \right)^{\delta_{i, 1}} 
  Y^W_{WV'} \left( \Phi_i({\rm v}_{n_i+1+ \delta_{i, 1}}  \otimes  \ldots \otimes 
{\rm v}_{n_i+k_i-r_i+\delta_{i, l}} \otimes 
  u_n)
 \right.  &&
\nn
\left. \left. 
(z_{n_i+1+ \delta_{i, 1}}, \ldots, {\rm z}_{n_i+k_i-r_i+\delta_{i, l}}, \zeta_{1, i}\right),  
\zeta_{2, i} \right)\; \overline{u}_n   \rangle     &&
\end{eqnarray*}
\begin{eqnarray*}
 +   R (-1)^{k+1}    
 \sum_{u_n \in V_{(n)} \atop n \in \Z}    
\prod_{i=1}^l
  \epsilon^n   
\langle w, \left(Y_W({\rm v}_{n_l+k_l-r_l+1}, z_{n_l+k_l-r_l+1}) \right)^{\delta_{i,l}} 
\qquad \qquad 
&& 
\nn
  Y^W_{WV'}  
 \left(  \Phi_i({\rm v}_{n_i+1} \otimes  \ldots 
 \otimes {\rm v}_{n_i+k_i-r_1} \otimes  
  u_n)  
 \left. 
 (z_{n_i+1}, \ldots, {\rm z}_{n_i+k_i-r_i},  \zeta_{1, i}\right),  
\zeta_{2, i} \right)\; \overline{u}_n  \rangle.  && 
\end{eqnarray*}
Consider the second and third terms (similar to \eqref{paros}) 
 in the action of $\delta^{k_i-r_i}_{m_i-t_i}$ operators 
 of the right hand side of \eqref{leibniz0}
\begin{eqnarray*}
&& 
 \sum\limits_{q=1}^l (-1)^{q+1} R \sum_{u_n\in V_{(n)} \atop n \in \Z}     
\prod_{i=1}^l
 \epsilon^n \langle w, \left( Y_W({\rm v}_{n_i+1}, z_{n_i+1}) \right)^{\delta_{q, i}}
\nn
&&
\qquad  Y^W_{WV'} \left(  
\Phi_i({\rm v}_{n_i+2}\otimes  \ldots 
 \otimes {\rm v}_{n_i+k_i-r_i}\otimes  
  u_n)(z_{n_i+2}, \ldots, z_{n_i+k_i-r_i}, \zeta_{1, i}\right), 
\zeta_{2, i} ) \; \overline{u}_n   \rangle.     
\end{eqnarray*}
This is equivalent to     
\begin{eqnarray*}
&& \sum\limits_{w \in W}  
 \sum\limits_{q=1}^l (-1)^{q+1}  R \sum_{u_n\in V_{(n)} \atop n \in \Z}     
\prod_{i=1}^l
 \epsilon^n \langle w, \left( Y_W({\rm v}_{n_i+1}, z_{n_i+1}) \right)^{\delta_{q, i}} w \rangle 
\nn
&&
  \langle w', Y^W_{WV'} \left(
\Phi_i({\rm v}_{n_i+2}\otimes  \ldots 
 \otimes {\rm v}_{n_i+k_i-r_i}\otimes  
  u_n)(z_{n_i+2}, \ldots, z_{n_i+k_i-r_i}, \zeta_{1, i}\right), 
\zeta_{2, i} ) \; \overline{u}_n   \rangle     
\end{eqnarray*}
\begin{eqnarray*}
&& =\sum\limits_{w \in W}  
 \sum\limits_{q=1}^{l-1} (-1)^{q+1}  R \sum_{u_n\in V_{(n)} \atop n \in \Z}     
\prod_{i=1}^l
 \epsilon^n \langle w, \left( Y_W({\rm v}_{n_i+1}, z_{n_i+1}) \right)^{\delta_{q, i}} w \rangle 
\nn
&&
 \langle w', Y^W_{WV'} \left(  
\Phi_{i-1}({\rm v}_{n_{i-1}+1}\otimes  \ldots 
 \otimes {\rm v}_{n_{i-1}+k_{i-1}-r_{i-1}}\otimes  
  u_n)
\right. 
\nn
&& \left. \qquad \qquad \qquad \qquad 
(z_{n_{i-1}+1}, \ldots, z_{n_{i-1}+k_{i-1}-r_{i-1}}, \zeta_{1, {i-1}}\right), 
\zeta_{2, {i-1}} ) \; \overline{u}_n   \rangle     
\end{eqnarray*}
\begin{eqnarray*}
&& =\sum\limits_{w \in W}  
 \sum\limits_{q=1}^{l-1} (-1)^{q+1} R \sum_{u_n\in V_{(n)} \atop n \in \Z}     
\prod_{i=1}^l
 \epsilon^n \langle w, 
\left( Y_W({\rm v}_{n_{i-1}+k_{i-1}-r_{i-1}+1}, z_{n_{i-1}+k_{i-1}+1}) \right)^{\delta_{q, i}} w \rangle 
\nn
&&
 \langle w', Y^W_{WV'} \left(  
\Phi_{i-1}({\rm v}_{n_{i-1}+1}\otimes  \ldots 
 \otimes {\rm v}_{n_{i-1}+k_{i-1}-r_{i-1}}\otimes  
  u_n) \right. 
\nn
&&
\qquad \qquad \qquad \qquad \left. 
(z_{n_{i-1}+1}, \ldots, z_{n_{i-1}+k_{i-1}-r_{i-1}}, \zeta_{1, {i-1}} \right), 
\zeta_{2, {i-1}} ) \; \overline{u}_n \rangle.      
\end{eqnarray*}
\begin{eqnarray*}
&& =  
 \sum\limits_{q=1}^l (-1)^{q+1}  R \sum_{u_n\in V_{(n)} \atop n \in \Z}     
\prod_{i=1}^l
 \epsilon^n \langle w, 
\left( Y_W({\rm v}_{n_i+k_i-r_i+1}, z_{n_i+k_i-r_i+1}) \right)^{\delta_{q, i}} 
\nn
&&
\qquad    Y^W_{WV'} \left( 
\Phi_i({\rm v}_{n_i+1}\otimes  \ldots 
 \otimes {\rm v}_{n_i+k_i-r_i}\otimes  
  u_n)(z_{n_i+1}, \ldots, z_{n_i+k_i-r_i}, \zeta_{1, i}\right), 
\zeta_{2, i} ) \; \overline{u}_n   \rangle     
\end{eqnarray*}
By combining with the third term of the right hand side of \eqref{leibniz0} 
which gives the opposite sign expression except for 
\begin{eqnarray*}
&& \sum\limits_{w \in W}  
   R \sum_{u_n\in V_{(n)} \atop n \in \Z}     
\prod_{i=1}^l
 \epsilon^n \langle w, \left( 
 \left( Y_W({\rm v}_1, z_1) \right)^{\delta_{i,1}}
  Y^W_{WV'} \left(  \Phi_i({\rm v}_{n_i+1+\delta_{i,1}} \otimes  \ldots  
 \otimes {\rm v}_{n_i+k_i-r_i}\otimes  
  u_n)
  \right. \right. 
\nn
&&
\qquad \qquad \qquad \qquad  \qquad  \qquad \qquad  \qquad  \left. 
(z_{n_i+1+\delta_{i,1}}, \ldots, z_{n_i+k_i-r_i}, \zeta_{1, i}), 
\zeta_{2, i} \right) \; \overline{u}_n  
\nn
&&
+(-1)^k 
 \left(Y_W({\rm v}_{k+1}, z_{k+1}) \right)^{\delta_{i,k}}  
  Y^W_{WV'} \left(  \Phi_i({\rm v}_{n_i+1} \otimes  \ldots  
 \otimes {\rm v}_{n_i+k_i-r_i}\otimes  
  u_n) \right. 
\nn
&&
\qquad \qquad \qquad  \qquad  \qquad  \qquad \qquad  \qquad  \left. \left. 
(z_{n_i+1}, \ldots, z_{n_i+k_i-r_i}, \zeta_{1, i}), 
\zeta_{2, i} \right) \; \overline{u}_n \right)  
 \rangle, 
\end{eqnarray*}
and including the first term, 
we obtain the action of $\delta^{k_i-r_i}_{m_i-t_i}$ on $\Phi_i$.  
The statement of the proposition for $\delta^{2, i}_{ex, i}$
 \eqref{halfdelta} can be checked accordingly. 
\end{proof}
\section{Bigraded differential algebras associated to a vertex algebra}
\label{invariants}
In this and the next Sections we provide the main results of the paper by deriving 
relations for the double graded differential algebras 
associated to chain bicomplexes, in particular, to bicomplexes \eqref{bigcomplex} and \eqref{shortseq}
for a grading-restricted vertex algebra.   
In analogy with the notion of integrability for differential forms on foliated manifolds 
\cite{Ghys},   
 we introduce here the notion of orthogonality for elements of spaces of a general complex.   
Let us consider a chain complex  
\begin{equation}
\label{doublecomplex}
\cdots \stackrel{\delta_{-2}}{\longrightarrow} C_{-1} \stackrel{\delta_{-1}}{\longrightarrow} C_0  
\stackrel{\delta_0}{\longrightarrow} 
C_1 
\stackrel{\delta_1}{\longrightarrow}C_2  \stackrel{\delta_2}{\longrightarrow} 
\cdots \stackrel{\delta_{n-1}}{\longrightarrow}C_n   \stackrel{\delta_n}{\longrightarrow}   \cdots , 
\end{equation}
with appropriate operators $\delta_n$, $n \ge 0$. 
For any choice of a fixed $l$-set $\mathcal L$ of ordered arbitrary elements 
 $\phi_{q_i} \in C_{q_i}$, $q_i \in \Z$,  $1 \le i \le l$ 
we introduce the commutators $\phi_p \cdot_l \phi_q$,   $p$, $q \in \mathcal L$, 
of a pair of elements $\phi_p \in C_p$, $\phi_q \in C_q$,  
 with respect to $*_l$-multiplications, as 
 \begin{eqnarray}
\label{defmultiplication}
\label{multiplication}
 \phi_p \cdot_l \phi_q = *_l( \phi_1, \ldots, \phi_p, \phi_q, \ldots, \phi_l)  
- *_l( \phi_1, \ldots, \phi_q, \phi_p, \ldots, \phi_l), &&     
\end{eqnarray}
with respect to 
 a general $*_l$-multiplication $*_l(., \ldots, . )$
of $l$ elements of a chain complex. 
In what follows we skip fixed elements 
 $\phi_1, \ldots, \phi_l$ in the notation $\phi_p \cdot_l \phi_q$.  
\begin{definition}
\label{turbo}
 Let us require that for a pair of a bicomplex spaces $C_p$ and $C_q$, $p$, $q \ge 0$, 
 there exist subspaces $C'_p \subset C_q$  
 such that for $\phi_p \in C'_p$ and  
$\phi_q \in C'_q$,  
\begin{equation}
\label{ortho}
\phi_p \cdot_l \delta_q \phi_q=0, 
\end{equation}
i.e., 
$\phi_p$ commutes to $\delta_q\phi_q$ with respect to the commutation \eqref{defmultiplication}).  
We call \eqref{ortho} the $l$-th orthogonality condition for mappings of the complex \eqref{doublecomplex}.  
\end{definition}
The double grading condition on elements of a general bicomplex spaces  
occurs from 
the assumption that elements of 
both sides of equations following from orthogonality
 condition 
 belong to the same bicomplex 
space.  
We formulate the first main result of this paper: 
\begin{theorem}
\label{dubina}
The orthogonality condition \eqref{ortho} endows
elements of a general bicomplex spaces $\delta^n_m: C^n_m\to C^{n+1}_{m_1}$, $n$, $m \ge0$, 
with the structure of a multiple bigraded differential   
algebra with respect to general $\cdot_l$-multiplications, $l \ge 0$,  
linear with respect to the additive group.    
\end{theorem}
\begin{proof}
Let us consider the most general case. 
 For non-negative $n_0$, $n$, $n_1$, $m_0$, $m$, $m_1$,  
let  $\chi \in C^{n_0}_{m_0}$,  $\Phi \in C^n_m$,   
and $\alpha \in C^{n_1}_{m_1}$.  
For $\Phi$ and $\alpha$, let $r_0$ be the number of common vertex algebra elements (and formal parameters), 
and $t_0$ be the number of common vertex operators $\Phi$ and $\alpha$ are composable to. 
Note that we assume $n$, $n_1 \ge r_0$,  $m$, $m_1 \ge t_0$.  
 Taking into account the orthogonality condition 
\[
\Phi \cdot_l \delta^{n_0}_{m_0} \chi=0, 
\]
implies that there exist $\alpha_1 \in C^{n_1}_{m_1}$, 
such that  
\[
 \delta^{n_0}_{m_0} \chi= \Phi \cdot_l \alpha_1.  
\]
From the last equations we obtain 
\[
n_0+1=n+n_1-r_0,
\]
\[
m_0-1=m+m_1-t_0.
\] 
Note that we have extra conditions following from the last identities: $n_0+1 \ge 0$,  $m_0-1 \ge 0$. 
The conditions above for indexes express the double grading condition 
for the bicomplex.    
As a result, we have a system in integer variables satisfying the grading conditions above.  
Consequently acting by corresponding coboundary operators 
we obtain the full structure of differential relations  
\begin{eqnarray}
\label{ogromno}
&& \Phi \cdot_l \delta^{n_0}_{m_0} \chi=0, 
\nn
&&
\delta^{n_0}_{m_0} \chi= \Phi \cdot_l \alpha_1, 
\nn
&&0=\delta^n_m \Phi \cdot_l \alpha_1 + (-1)^n \Phi \cdot_l \delta^{n_1}_{m_1} \alpha_1, 
\nn
&& 
0=\delta^n_m \Phi \cdot_l \delta^{n_0}_{m_0} \chi, 
\nn
&&
\delta^{n_0}_{m_0} \chi = \delta^n_m \Phi \cdot_l \alpha_2, 
\nn
&& 0=\delta^n_m \Phi \cdot_l \delta^{n_i}_{m_i} \alpha_i, \; \;
\nn
&&\delta^{n_i}_{m_i} \alpha_i = \delta^n_m \Phi \cdot_l \alpha_{i+1}, \; \;
i \ge 2,
\end{eqnarray}
where $\alpha_i \in C^{n_i}_{m_i}$,   
and 
$n_i$, $m_i$, $i\ge 2$ satisfy relations  
\[
n_i=n+n_{i+1}-r_{i+1},
\] 
\[
m_i = m +m_{i+1}-t_{i+1}.
\] 
The sequence of relations \eqref{ogromno} does not cancel 
until the conditions on indexes given above fulfill.  
\end{proof} 
Next, we have the second main result of this paper 
\begin{proposition}
The set \eqref{ogromno} of commutation relations generates a sequence 
of non-vanishing cohomology invariants 
\[
\bigoplus_{p, q=n_0, m_0; n, m; n_i, m_i}  
\left\{ \left[\left(\delta^{n_0}_{m_0} \chi \right)\cdot_l \chi \right], \; 
\left[\left(\delta^n_m \Phi \right)\cdot_l \Phi \right],  \; 
\left[\left(\delta^{n_i}_{n_i} \alpha_i \right)\cdot_l \alpha_i \right] \right\},    
\]
 for $i=1, \ldots, L$, for some $L \in \N$, with 
non-vanishing $\left(\delta^{n_0}_{m_0} \chi \right)\cdot_l \chi$,  
$\left(\delta^n_m \Phi \right)\cdot_l \Phi$, 
 and 
$\left(\delta^{n_i}_{m_i} \alpha_i \right)\cdot_l \alpha_i$.   
 These classes are independent on the choices of $\chi \in C^{n_0}_{m_0}$, 
 $\Phi \in C^n_m$, and 
  $\alpha_i \in C^{n_i}_{m_i}$.   
\end{proposition}
\begin{proof}
Let $\phi$ be one of generators $\chi$, $\Phi$, $\alpha_i$, $\beta$, $1 \le i \le L$. 
Let us show now the non-vanishing property of $\left(\left(\delta^n_m \phi \right)\cdot_l \phi\right)$.  
Indeed, suppose 
\[
\left(\delta^n_m \phi \right)\cdot_l \phi=0.
\]
 Then there exists $\gamma \in C^{n'}_{m'}$,   
such that 
\[
\delta^n_m \phi =\gamma \cdot_l \phi.
\]
 Both sides of the last equality should belong to the same bicomplex 
space but one can see that it is not possible since we obtain $m'=t-1$,
 i.e., the number of common vertex operators 
for the last equation is greater than for one of multipliers.  
Thus, $\left(\delta^n_m \phi \right)\cdot_l \phi$ is non-vanishing.  

Now let us show  that $\left[\left(\delta^n_m \phi \right)\cdot_l \phi \right]$   
 is invariant, i.e., it does not depend on the choice of $\Phi \in C^n_m$.  
 Substitute $\phi$ by 
$\left(\phi + \eta\right)\in C^n_m$.     
We have 
\begin{eqnarray}
\label{pokaz}
\nonumber
\left(\delta^n_m \left( \phi + \eta \right) \right) \cdot_l \left( \phi  + \eta \right) &=& 
\left(\delta^n_m \phi\right) \cdot_l \phi  
 + \left( \left(\delta^n_m \phi \right)\cdot_l \eta 
-  \phi \cdot_l \delta^n_m \eta  \right)   
\nn
&+& \left( \phi \cdot_l \delta^n_m \eta   +  
\delta^n_m \eta  \cdot_l \phi \right) 
 +
\left(\delta^n_m \eta \right) \cdot_l \eta.  
\end{eqnarray}
Since
\begin{eqnarray*}
 \left( \phi \cdot_l \delta^n_m \eta  + 
\left(\delta^n_m \eta\right)  \cdot_l \phi \right)
 &=& *_l(\phi_1,  \ldots, \phi \delta^n_m \eta, \ldots, \phi_l) 
- *_l(\phi_1, \ldots, (\delta^n_m \eta), \phi, \ldots, \phi_l)    
\nn
&+& *_l (\phi_1, \ldots, \left(\delta^n_m \eta\right),  \phi, \ldots, \phi_l) 
- *_l(\phi_1, \ldots, \phi,  \delta^n_m \eta, \ldots, \phi_l)=0,    
\end{eqnarray*}
then \eqref{pokaz} represents the same cohomology class 
$\left[ \left(\delta^n_m  \phi \right) \cdot_l \phi \right]$. 
\end{proof}
 In particular, the orthogonality condition for the bicomplexes \eqref{bigcomplex} and \eqref{shortseq} 
 together with the action of coboundary operators $\delta^n_m$ and $\delta^2_{\frac{1}{2}}$,
 and the multiplication \eqref{Z2n_pt_epsss},   
defines a differential algebra depending on vertex algebra elements and formal parameters, 
and provides cohomology invariants described above. 
\section{Examples}
In this Section we consider particularly interesting examples of algebras
and their invariants 
 described in Proposition \ref{generators}. 
We restrict ourselves to the case of the $*_l$-multiplication     
 introduced in Section \ref{multiplication}  
for elements of 
the spaces $C_{m_i-t_i}^{k_i-r_i}$. 
 The orthogonality condition for a bicomplex sequence \eqref{shortseq}, 
 together with the action of coboundary operators $\delta^n_m$ and $\delta^2_{\frac{1}{2}}$, 
and the multiplication \eqref{Z2n_pt_epsss},    
define a differential bigraded algebra depending on vertex algebra elements and formal parameters. 
In particular, for the bicomplex \eqref{bigcomplex}, we obtain in this way the generators and commutation relations 
for a continual Lie algebra $\mathcal G(V)$ (a generalization of ordinary Lie algebras with continual  
space of roots, c.f. \cite{saver} described in Appendix \eqref{tusar})  
with the continual root space represented by a grading-restricted vertex algebra $V$. 
\subsection{Invariants associated with $C^2_{\frac{1}{2}}$}  
Due to non-trivial action of the coboundary operators 
\[
\delta^2_{\frac{1}{2}}: C^2_{\frac{1}{2}}  
\to C^3_0, 
\]
 the case $\Phi \in C^2_{\frac{1}{2}}$ is exceptional among the relations 
coming from the double grading condition for corresponding vertex algebra bicomplex, 
 and allows to reconstruct the classical invariants.  
In this subsection we consider the case $l=2$, and denote the multiplication $\cdot=\cdot_2$.  
  Let $\Phi \in C^2_{\frac{1}{2}}$ and $\chi \in C^n_m$.   
Then we require the orthogonality: 
\[
\chi \cdot \delta^2_{\frac{1}{2}} \Phi=0.  
\]
Thus, there exist $\beta \in C^{n'}_{m'}$  
such that 
\[
 \delta^2_{\frac{1}{2}} \Phi= \chi \cdot \beta. 
\]
We then get  
\[
3=n+n'-r,
\]
\[
 0=m+m'-t. 
\]
 Let $n=r+\alpha$, $0 \le \alpha \le 3$, $n'=3-\alpha$; $m'=t-m \ge 0$, i.e., $t=m$, thus, $m=t=m'=0$.
Thus, $\chi \in C^{r+\alpha}_0$, $\beta \in C^{3-\alpha}_0$.    
For $r+\alpha=3-\alpha=2$ we obtain $\alpha=1$, $r=1$. If we require $\chi=\Phi \in C^2_k 
\subset  C^2_{\frac{1}{2}}$, $k >0$,   
then the equation 
\[
\delta^2_{\frac{1}{2}} \Phi= \Phi \cdot \beta,
\]
 corresponds to a generalization of Godbillon-Vey invariant \cite{Ghys} for differential forms.
We obtain also the commutation relations: 
\[
\left[H, 
X_{+2}
\right]=0, 
\]
\[
\left[ H, 
Y_-
\right]=X_{+2},
\] 
\[
\left[ X_{-2}, 
X_{+2} 
\right]=H,
\]
\[
\left[ Y_+, H_0 \right]=X_{+2},
\]
for generators 
\[
H
=\chi
, 
\]
\[
X_{+1}
=\Phi
,
\] 
\[
X_{+2}
= \delta^2_{\frac{1}{2}} \Phi
,
\] 
\[
Y_- 
=\beta 
.
\] 
 It is easy to see that since all mappings have zero operators composable with, then all further 
actions of the coboundary operators vanish.
Nevertheless, recall that 
\[
C^2_k \subset  C^2_{\frac{1}{2}}, 
\]
 $k > 0$, thus, 
we can consider the most general case when 
$\chi \in C^{r+\alpha}_{m_0}$, $\Phi \in C^2_{m_1}$, 
 $\beta \in C^{3-\alpha}_{m_2}$. 
Then the grading condition requires $m_1-1=m_0 + m_2 - t'$, where $t'$ is the number of common vertex operators 
for $\chi \in C^{r+\alpha}_{m_0}$ and $\beta \in C^{3-\alpha}_{m_2}$.  
Thus, on acting by coboundary operators we obtain further commutation relations of the form \eqref{ogromno}.  
\subsection{A continual Lie algebra associated to the bicomplex \eqref{bigcomplex}}
Using the orthogonality condition, 
  way we obtain the generators and commutation relations  
for a continual Lie algebra.  
We have the next result of the paper. 
\begin{proposition}
\label{generators}
For the bicomplex \eqref{bigcomplex}, the generators
\[
\left\{
\chi \right., 
\Phi,
 \alpha_i, 
\; 
\delta^{n_0}_{m_0} \chi, 
 \delta^{n_1}_{m_1} \Phi, 
\\
  \left. \delta^{n_i}_{m_i} \alpha_i
\right\}, 
\]
with $i \ge 0$, 
 and commutation relations \eqref{ogromno}  
form a continual Lie algebra $\mathcal G(V)$ with a root space provided by the grading-restricted 
vertex algebra $V$.  
\end{proposition}
\begin{proof}
With the redefinition (we suppress here the dependence on vertex algebra elements and formal parameters), 
\[
H_0=\delta^0_3\chi, 
\]
\[
H^*_0=\chi,
\] 
\[
X_{+1}=\Phi,
\]
\[
X_{-i}=\alpha_i,
\]
\[ 
 Y_{+1}=\delta^1_2\Phi,
\] 
\[
 Y_{-i}=\delta^1_t\alpha_i,
\]  
we arrive at the commutation relations: 
\begin{eqnarray*}
&& \left[H_0, X_{+1}\right]=0, \; \;  
\nn
&&
\left[ X_{+1}, X_{-1}\right]=H_0,    
\\
&&
\left[ Y_{+1}, X_{-1}\right]= 
 (-1)^n \left[Y_{-1}, X_{+1}\right], \; \;    
\nn
&&
\left[Y_{+1}, X_{-2} \right]=0,   
\\
&&\left[ Y_{+1}, Y_{-i} \right]=0, \; \; 
\nn
&&
\left[ Y_{+1}, X_{-(i+1)}\right]=Y_{-i}. 
\end{eqnarray*} 
One easily checks Jacobi identities for generators.   
\end{proof}
\section{Appendix: the grading-restricted vertex algebra}
\label{grading}
In this Section, we recall \cite{Huang} properties of  
grading-restricted vertex algebras and their grading-restricted generalized 
modules over the base field $\C$ of complex numbers.   
A vertex algebra  
$(V,Y_V,\mathbf{1}_V)$, cf. \cite{K, BZF},  consists of a $\Z$-graded complex vector space 
\[
V = 
\bigoplus_{n\in\Z}\,V_{(n)}, 
\]
\[
\dim V_{(n)}<\infty, 
\]
for each 
$n\in \Z$,  
and linear map 
\[
Y_V:V\rightarrow {\rm End \;}(V)[[z,z^{-1}]], 
\]
 for a formal parameter $z$ and a 
distinguished vector $\mathbf{1}_V\in V$.  
The evaluation of $Y_V$ on $v\in V$ is the vertex operator
\[
Y_V(v,z) = \sum_{n\in\Z}v(n)z^{-n-1}, 
\]
with components
\[
(Y_V(v))_n=v(n)\in {\rm End \;}(V),
\]
 where 
\[
Y_V(v,z)\mathbf{1} = v+O(z).
\]
A grading-restricted vertex algebra is subject to the following 
\noindent
\begin{enumerate}
\item 
\noindent
 Grading-restriction condition:
$V_{(n)}$ is finite dimensional for all $n\in \Z$, and $V_{(n)}=0$, for $n\ll 0$. 

\item { Lower-truncation condition}:
For $u, v\in V$, $Y_V(u, z)v$ contains only finitely many negative 
power terms, i.e., $Y_V(u, z)v\in V((z))$ (the space of formal 
Laurent series in $z$ with coefficients in $V$).   

\item  Identity property: 
Let ${\rm Id}_V$ be the identity operator on $V$. Then 
\[
Y_V(\mathbf{1}_V, z)={\rm Id}_V.
\] 

\item  Creation property: For $u\in V$, $Y_V(u, z)\mathbf{1}_V\in V[[z]]$, 
and 
\[
\lim_{z\to 0}Y_V(u, z)\mathbf{1}_V=u.
\]

\item  Duality: For $u_1$, $u_2$, $v\in V$, 
$v'\in V'=\coprod_{n\in \mathbb{Z}}V_{(n)}^*$ ($V_{(n)}^*$ denotes
the dual vector space to $V_{(n)}$ and $\langle\,. ,. \rangle$ the evaluation 
pairing $V'\otimes V\to \C$), the series 
$\langle v', Y_V(u_2, z_2)Y_V(u_1, z_1)v\rangle$, and  
$\langle v', Y_V(Y_V(u_1, z_1 -z_2)u_2, z_2 )v\rangle$, 
are absolutely convergent
in the regions $|z_1|>|z_2|>0$, $|z_2|>|z_1|>0$,
$|z_2|>|z_1 -z_2|>0$, respectively, to a common rational function 
in $z_1$ and $z_2$ with the only possible poles at $z_1=0=z_2$ and 
$z_1=z_2$. 

One assumes the existence of Virasoro vector $\omega\in V$:
its vertex operator 
\[
Y(\omega, z)=\sum_{n\in\Z}L(n)z^{-n-2}, 
\]
is determined by Virasoro operators $L(n): V\to V$ fulfilling 
\[
[L(m), L(n)]=(m-n)L(m+n)+\frac{c}{12}(m^{3}-m)\;\delta_{m+b, 0}\; {\rm Id_V},
\]
($c$ is called the central charge of $V$).   
The grading operator is given by 
\[
L(0)u=nu, \quad u\in V_{(n)}, 
\]
($n$ is called the weight of $u$ and denoted by $\wt(u)$).
\item { $L_V(0)$-bracket formula}: Let $L_V(0): V\to V$ 
be defined by $L_V(0)v=nv$ for $v\in V_{(n)}$. Then
\[
[L_V(0), Y_V(v, z)]=Y_V(L_V(0)v, z)+z\frac{d}{dz}Y_V(v, z),
\]
for $v\in V$.
\item { $L_V(-1)$-derivative property}: 
Let $L_V(-1): V\to V$ be the operator 
given by 
\[
L_V(-1)v=\res_z z^{-2}Y_V(v, z)\one=Y_{(-2)}(v)\one,
\]
for $v\in V$. Then for $v\in V$, 
\[
\frac{d}{dz}Y_V(u, z)=Y_V(L_V(-1)u, z)=[L_V(-1), Y_V(u, z)].
\]
\end{enumerate}
A grading-restricted generalized $V$-module is a vector space 
$W$ equipped with a vertex operator map 
\[
Y_W: V\otimes W\to W[[z, z^{-1}]],
\] 
\[
u\otimes w\mapsto  Y_W(u, w)\equiv Y_W(u, z)w=\sum_{n\in \Z}(Y_W)_{n}(u,w)z^{-n-1}, 
\]
and linear operators $L_W(0)$ and $L_W(-1)$ on $W$, satisfying conditions similar as in the 
definition for a grading-restricted vertex algebra. In particular, 
\begin{enumerate}
\item {Grading-restriction condition}:
The vector space $W$ is $\mathbb C$-graded, i.e., 
$W=\coprod_{\alpha\in \mathbb{C}} W_{(\alpha)}$, such that 
$W_{(\alpha)}=0$ when the real part of $\alpha$ is sufficiently negative. 

\item { Lower-truncation condition}:
For $u\in V$ and $w\in W$, $Y_W(u, z)w$ contains only finitely many negative 
power terms, i.e., $Y_W(u, z)w\in W((z))$.

\item { Identity property}: 
Let ${\rm Id}_W$ be the identity operator on $W$,
$Y_W(\mathbf{1}_V, z)={\rm Id}_W$.

\item { Duality}: For $u_1$, $u_2\in V$, $w\in W$,
$w'\in W'=\coprod_{n\in \mathbb{Z}} W_{(n)}^*$ ($W'$ is 
the dual $V$-module to $W$), the series 
 $\langle w', Y_W(u_1, z_1)Y_W(u_2, z_2)w\rangle$,
$\langle w', Y_W(u_2, z_2)Y_W(u_1, z_1)w\rangle$, and  
$\langle w', Y_W(Y_V(u_1, z_1 -z_2 )u_2, z_2 )w\rangle$, 
are absolutely convergent
in the regions $|z_1|>|z_2|>0$, $|z_2|>|z_1|>0$,
$|z_2|>|z_1 -z_2|>0$, respectively, to a common rational function  
in $z_1$ and $z_2$ with the only possible poles at $z_1=0=z_2$ and 
$z_1=z_2$. 
\item { $L_W (0)$-bracket formula}: For  $v\in V$,
$[L_W(0), Y_W(v, z)]=Y_W(L(0)v, z)+z\frac{d}{dz}Y_W(v, z)$. 
\item { $L_W(0)$-grading property}: For $w\in W_{(\alpha)}$, there exists
$N\in \Z_+$ such that $(L_W(0)-\alpha)^N w=0$. 
\item { $L_W(-1)$-derivative property}: For $v\in V$,  
$\frac{d}{dz}Y_W(u, z)=Y_{W}(L_{V}(-1)u, z)=[L_W(-1), Y_W(u, z)]$.
\end{enumerate}
 A graded-restricted vertex algebra is endowed with the unique symmetric invertible 
invariant bilinear form $\langle ., .  \rangle$
 with normalization 
\[
\langle\vac, \vac \rangle=1,  
\] 
where \cite{FHL, Li}  
\begin{equation}
\label{formain}
\langle Y^{\dagger}(a,z) \; b, c \rangle = \left\langle b, Y (a,z)c \right\rangle, 
\end{equation}
for 
\begin{eqnarray}
Y^{\dagger}(a,z)
&=&\sum_{n\in\Z} a^{\dagger}(n)z^{-n-1}
\nn
&=&
Y\left(e^{zL_V(1)}\left(-z^{-2}\right)^{L_V(0)}a,z^{-1}\right).
\end{eqnarray}
\section{Appendix: Continual Lie algebras}
\label{tusar}
Continual Lie algebras were introduced in \cite{saver}
and then studied in \cite{sv2, sv3}.
 Suppose $\mathcal E$ is an associative algebra (which we call the
 base algebra) over $\R$ or $\C$, and 
\[
K_0, \; K_\pm, \; K_{0,0}: \mathcal E \times \mathcal E \to \mathcal E, 
\]
 are bilinear mappings.
 The local Lie part of a continual Lie algebra is defined as
\[
{\widehat{ \mathcal G}}=\mathcal G_{-1} \oplus \mathcal G_0 \oplus \mathcal G_{+1},
\]
 where $\mathcal G_i$,
$i=0, \pm 1$, are isomorphic to $\mathcal E$ and parametrized by its elements.
 The subspaces $\mathcal G_i$ consist of the elements
\[
\left\{ X_i(\phi), \phi \in \mathcal E \right\}, i=0, \pm 1.
\] 
The generators 
$X_i(\phi)$ are subject to the commutation relations
\[
\left[  X_0(\phi), X_0(\psi) \right]=X_0(K_{0,0}(\phi, \psi)), 
\]
\[
\left[ X_0(\phi), X_{\pm 1}(\psi) \right]= X_{\pm 1}(K_\pm(\phi, \psi)), 
\]
\[
\left[ X_{+1}(\phi), X_{-1}(\psi) \right]= X_0(K_0(\phi, \psi)),
\] 
for all $\phi$, $\psi \in \mathcal E$.
 It is also assumed that Jacobi identities are satisfied. Then the  
 conditions on mappings $K_{0,0}$, $K_{0,\;  \pm}$ follow: 
\[
K_\pm(K_{0,0}( \phi, \psi) , \chi)=
K_\pm(\phi, K_\pm(\psi, \chi)) - K_\pm(\psi, K_\pm(\phi, \chi)),
\]
\[
K_{0,0}( \psi, K_0(\phi, \chi)) =
K_0(K_+(\psi, \phi), \chi)) + K_0(\phi, K_-(\psi, \chi)),
\]
for all $\phi$, $\psi$, $\chi \in \mathcal E$.
An infinite dimensional algebra 
\[
\mathcal G(\mathcal E; K) = \mathcal G'(\mathcal E; K)/J,  
\]
is called a continual contragredient Lie algebra, 
where $\mathcal G'(\mathcal E; K)$ is a Lie algebra
 freely generated by ${\widehat{ \mathcal G}}$, and $J$
is the largest homogeneous ideal with trivial intersection with $\mathcal G_0$ 
(consideration of the quotient is equivalent to imposing the Serre
relations in an ordinary Lie algebra case) \cite{sv2, sv3}.
\section*{Acknowledgement} 
The author's research was supported by the GACR project 18-00496S and RVO: 67985840. 


\begin{thebibliography}{99}
\bibitem
{BG}  Bazaikin, Ya. V.,  Galaev, A. S.  Losik classes for codimension one foliations, 
 J. Inst. Math. Jussieu 21 (2022), no. 4, 1391--1419. 
\bibitem
{BGG} Ya. V. Bazaikin, A. S. Galaev, and P. Gumenyuk. 
Non-diffeomorphic Reeb foliations and modified Godbillon-Vey class, 
Math. Z. (2022), no. 2, 1335--1349.


\bibitem 
{BZF} 
E. Frenkel, D. Ben-Zvi. Vertex algebras and algebraic curves. Mathematical Surveys and Monographs,
 88. American Mathematical Society, Providence, RI, 2001. xii+348 pp. 


\bibitem
{FHL}
I.~B. Frenkel, Y.-Z. Huang and J.~Lepowsky. 
On axiomatic approaches to vertex operator algebras and modules,
preprint, 1989;
{\em Memoirs Amer. Math. Soc.} {\bf 104}, 1993.
 

\bibitem
{FMS} Ph. Francesco, P. Mathieu, and D. Senechal. Conformal Field Theory. 
 Graduate Texts in Contemporary Physics. 1997. 

\bibitem{FQ} F. Qi.    
Representation theory
and cohomology theory of meromorphic open string vertex algebras, Ph.D. dissertation, (2018). 

\bibitem
{Ghys} E. Ghys.     
L'invariant de Godbillon-Vey. Seminaire Bourbaki, 
41--eme annee, n 706, S. M. F.
Asterisque 177--178 (1989)


\bibitem
{Huang} Y.-Zh. Huang. 
A cohomology theory of grading-restricted vertex algebras. 
Comm. Math. Phys. 327 (2014), no. 1, 279--307. 

\bibitem
{Hu3} Y.-Z. Huang, The first and second cohomologies 
of grading-restricted vertex algebras, 
 Comm. Math. Phys. 327 (2014), no. 1, 261--278.  

\bibitem
{H2} Y.-Z. Huang. {\it Two-Dimensional Conformal Geometry
and Vertex Operator Algebras}, Progress in Mathematics, Vol. 148,
Birkh\"auser, Boston, 1997.

\bibitem
{K} V. Kac.  
 {\it Vertex Operator Algebras for Beginners}, 
University Lecture Series \textbf{10}, AMS, Providence 1998.


\bibitem
{Li} H. Li. 
Symmetric invariant bilinear forms on vertex operator algebras, 
J. Pure. Appl. Alg. \textbf{96} (1994)  279--297.


\bibitem{Losik} M. V. Losik.   
On some generalization of a manifold and its characteristic classes (Russian), Funcional. Anal. i 
Prilozhen.  
\textbf{24}(1990), no 1, 29-37 ;  English translation in Functional Anal. Appl. \textbf {24} (1990), 26--32. 

\bibitem
{saver} M. V. Saveliev. Integro-differential nonlinear equations and continual Lie algebras.
Comm. Math. Phys. 121 (1989), no. 2, 283--290. 


\bibitem{sv2} M. V. Saveliev,  A. M. Vershik. 
Continuum analogues of contragredient Lie algebras.
Commun. Math. Phys. 126, 367, 1989;

\bibitem{sv3} M. V. Saveliev,  A. M. Vershik. 
 New examples of continuum graded Lie algebras. Phys. Lett. A,
143, 121, 1990. 


\bibitem
{TUY} A. Tsuchiya, K. Ueno, and Y. Yamada, Y. Conformal field
theory on universal family of stable curves with gauge symmetries,
Adv. Stud. Pure. Math. \textbf{19} (1989), 459--566.

\bibitem
{TZ} M. P. Tuite, A. Zuevsky. A generalized vertex operator algebra for Heisenberg intertwiners.
 J. Pure Appl. Algebra 216 (2012), no. 6, 1442--1453.

\bibitem
{Y} A. Yamada. Precise variational formulas for abelian
differentials. Kodai Math.J. \textbf{3} (1980), 114--143.


\bibitem
{Zhu}
Y. Zhu. Modular invariance of characters of vertex operator algebras,
{\em J.
Amer. Math. Soc.} {\bf 9} (1996), 237--307.

\end{thebibliography}
\end{document}